\documentclass[12pt]{amsart}
\usepackage{amssymb}
\usepackage{amsthm}
\usepackage[margin=1in]{geometry}
\usepackage{amstext}
\usepackage{amsbsy}
\usepackage{amscd}
\usepackage{enumerate}
\usepackage{chngpage}
\usepackage{mathtools}
\usepackage{amsmath}
\usepackage{hyperref}
\usepackage{stmaryrd}
\usepackage{color}
\usepackage{enumitem}
\usepackage{lmodern}
\usepackage[T1]{fontenc}

\usepackage{tikz,tikz-qtree,tikz-qtree-compat,fullpage,adjustbox}
\usepackage{pst-plot}
\usepackage{etoolbox}
\usetikzlibrary{positioning}
\providecommand\circledcolorednumb{}\renewcommand\circledcolorednumb[2]{\tikz[baseline=(char.center)]{\node[shape = circle,draw, inner sep = 2pt,fill=#1](char)    {\phantom{00}};\node[anchor=center] at (char.center) {\makebox(0,0){\large{{\sf #2}}}};}}
\providecommand\dotscircles{}\renewcommand\dotscircles{{\bf \dots}}
\robustify{\circledcolorednumb}
\providecommand\nongap{}\renewcommand\nongap[1]{\circledcolorednumb{yellow}{#1}}
\providecommand\gap{}\renewcommand\gap[1]{\circledcolorednumb{black!20}{\phantom{#1}}}
\providecommand\generator{}\renewcommand\generator[1]{\circledcolorednumb{orange!80}{#1}}

\usepackage{tikz}

\DeclareMathOperator{\Ap}{Ap}
\DeclareMathOperator{\eg}{eg}

\newcommand{\Q}{{\mathbb Q}}
\newcommand{\Z}{{\mathbb Z}}
\newcommand{\N}{{\mathbb N}}
\newcommand{\R}{{\mathbb R}}
\newcommand{\T}{{\mathcal T}}
\newcommand{\Sg}{{\mathcal S}}

\theoremstyle{plain}\newtheorem{proposition}{Proposition}[section]
	\newtheorem{theorem}[proposition]{Theorem}
	\newtheorem{defn}[proposition]{Definition}
	\newtheorem{cor}[proposition]{Corollary}
	\newtheorem{lemma}[proposition]{Lemma}
	\newtheorem{conj}[proposition]{Conjecture}
	
	\newtheorem{question}[proposition]{Question}
	\newtheorem{prop}[proposition]{Proposition}
\newtheorem{thm}[proposition]{Theorem}
\newtheorem{example}[proposition]{Example}
	\newtheorem{remark}[proposition]{Remark}

\title{Ordinarization Numbers of Numerical Semigroups}
\author{Sogol Cyrusian}
\address{Department of Mathematics, University of California, Santa Barbara, CA 93106}
\email{sogol@ucsb.edu}
\author{Nathan Kaplan}
\address{Department of Mathematics, University of California, Irvine, CA 92697}
\email{nckaplan@math.uci.edu}
\date{\today}
\begin{document}

		\begin{abstract}
There has been significant recent interest in studying how the number of numerical semigroups of genus $g$ behaves as a function of $g$.  Bras-Amor\'os has shown how to organize the collection of numerical semigroups of genus $g$ into a rooted tree called the ordinarization tree.  The ordinarization number of a numerical semigroup $S$ is the length of the path from $S$ back to the root of the tree.  We study the problem of counting numerical semigroups of genus $g$ with a fixed ordinarization number $r$.  We show how this can be interpreted as a counting problem about integer points in a certain rational polyhedral cone and use ideas from Ehrhart theory to study this problem.  We give a formula for the number of numerical semigroups of genus $g$ and ordinarization number $2$, building on the corresponding result of Bras-Amor\'os for ordinarization number $1$.  We show that the ordinarization number of a numerical semigroup generated by two elements is equal to the number of integer points in a certain right triangle with rational vertices.  We consider the analogous problem for supersymmetric numerical semigroups with more generators.  We also study ordinarization numbers of numerical semigroups generated by an interval.
	\end{abstract}

	\maketitle

	\section{Introduction}

Let $\N_0 = \{0,1,2,\ldots\}$ be the set of nonnegative integers.  A \emph{numerical semigroup} $S \subseteq \N_0$ is an additive submonoid with finite complement in $\N_0$.  The elements of $\N_0 \setminus S$ are the \emph{gaps} of $S$.  The largest of these gaps is the \emph{Frobenius number} of $S$, denoted $F(S)$, and the number of these gaps is the \emph{genus} of $S$, denoted $g(S)$.  The \emph{multiplicity} of a numerical semigroup $S$, denoted $m(S)$, is its smallest nonzero element.  We recommend the book of Garc\'ia-S\'anchez and Rosales as a general reference for the theory of numerical semigroups \cite{GarciaSanchez_Rosales}.

Let $N(g)$ denote the number of numerical semigroups of genus $g$.  In an influential paper from 2008, Bras-Amor\'os computed the first $50$ values of $N(g)$ and made several conjectures about the growth of this function \cite{Bras-Amoros_fib}.  Zhai determined the asymptotic growth of this sequence, proving two of these conjectures.  Throughout this paper, we write $\varphi = \frac{1+\sqrt{5}}{2}$ for the Golden ratio.
\begin{thm}\cite[Theorem 1]{Zhai}\label{thm:Zhai}
There exists a positive real number $C$ such that 
\[
\lim_{g \rightarrow \infty} \frac{N(g)}{\varphi^g} = C.
\]
\end{thm}
\noindent More recently, Zhu has shown that $C > 3.8293$ and speculates that it is probably between $3.85$ and $3.86$ \cite[Proposition 7.2]{Zhu}.  For much more on this problem of counting numerical semigroups of given genus, see the survey of the second author \cite{KaplanSurvey}, papers of Bacher \cite{Bacher} and of Zhu \cite{Zhu}, and the paper of Delgado, Eliahou, and Fromentin where they compute $N(g)$ for all $g \le 75$ \cite{DelgadoEliahouFromentin}.  In a very recent preprint, Bras-Amor\'os describes a new algorithm to compute $N(g)$ and uses it to compute $N(76)$ and $N(77)$ \cite{BA_Exploring}.

The following conjecture from the original paper of Bras-Amor\'os remains open.
\begin{conj}[Bras-Amor\'os]\cite{Bras-Amoros_fib}
For all $g \ge 2$ we have $N(g) \ge N(g-1) + N(g-2)$.
\end{conj}
Even the following weaker version of this conjecture is not known.
\begin{conj}[Bras-Amor\'os]\label{Ng_inc}
For all $g \ge 1$ we have $N(g) \ge N(g-1)$.
\end{conj}
There has been significant interest in proving partial results towards this conjecture.  To give one notable example, Eliahou and Fromentin show that if you restrict to the subset of numerical semigroups of genus $g$ for which $F(S) < 3m(S)$, a subset that includes `most' numerical semigroups as $g \rightarrow \infty$, then a version of this conjecture is true \cite{EliahouFromentin}.

Several authors have proposed refined conjectures that would imply Conjecture \ref{Ng_inc}. For example, see the conjecture of the second author about the number of numerical semigroups of genus $g$ and fixed multiplicity \cite[Conjecture 7]{KaplanJPAA}, and the conjecture of Almeida and Bernardini about the number of gapsets of genus $g$ that are pure $\kappa$-sparse \cite[Section 6]{AlmeidaBernardini}.  Our paper is partially motivated by a  conjecture of Bras-Amor\'os about semigroups with given genus and ordinarization number. We recall some additional background and then state the conjecture.

Bras-Amor\'os introduced the ordinarization transform of a numerical semigroup \cite{Bras-Amoros_ord}.  The \emph{ordinary numerical semigroup of genus $g$} is $S_g = \{0,g+1,g+2,g+3,\ldots\}$.  This is the semigroup whose gaps are the positive integers that are less than or equal to $g$.  Let $S$ be a numerical semigroup of genus $g$ that is not equal to $S_g$.  It is clear that $S' = S \cup \{F(S)\} \setminus \{m(S)\}$ is also a numerical semigroup of genus $g$.  If $S' \neq S_g$, repeating this process gives another numerical semigroup $S'' = S' \cup \{F(S')\} \setminus \{m(S')\}$ that has genus $g$.  Noting that $F(S'') < F(S') < F(S)$ and $m(S'') > m(S') > m(S)$, we see that after finitely many iterations of this process we arrive at the semigroup $S_g$.  In this way, we can organize the set of all numerical semigroups of genus $g$ into a graph.  Since every semigroup of genus $g$ has a uniquely determined path back to the ordinary semigroup of genus $g$, this graph is a tree. Bras-Amor\'os calls this the \emph{ordinarization tree} of numerical semigroups of genus $g$ and denotes it by $\T_g$.

\begin{figure}\label{fig_T7}

{\begin{tikzpicture}[grow'=right, sibling distance=0.600000mm]\tikzset{level 1+/.style={level distance=14.000000cm}}\node (arbre) at (current page.north) {\adjustbox{max width=\textwidth,max height=.9\textheight} { \Tree[.{\begin{tabular}{c}{$\nongap{0}\gap{1}\gap{2}\gap{3}\gap{4}\gap{5}\gap{6}\gap{7}\generator{8}\generator{9}\generator{10}\generator{11}\generator{12}\generator{13}\generator{14}\generator{15}\dotscircles$} \\\end{tabular}} [.{\begin{tabular}{c}{$\nongap{0}\gap{1}\gap{2}\gap{3}\gap{4}\nongap{5}\gap{6}\gap{7}\gap{8}\generator{9}\nongap{10}\generator{11}\generator{12}\generator{13}\nongap{14}\nongap{15}\dotscircles$} \\\end{tabular}} ][.{\begin{tabular}{c}{$\nongap{0}\gap{1}\gap{2}\gap{3}\gap{4}\gap{5}\nongap{6}\gap{7}\gap{8}\generator{9}\generator{10}\generator{11}\nongap{12}\generator{13}\generator{14}\nongap{15}\dotscircles$} \\\end{tabular}} [.{\begin{tabular}{c}{$\nongap{0}\gap{1}\gap{2}\gap{3}\gap{4}\nongap{5}\nongap{6}\gap{7}\gap{8}\gap{9}\nongap{10}\nongap{11}\nongap{12}\generator{13}\generator{14}\nongap{15}\dotscircles$} \\\end{tabular}} ][.{\begin{tabular}{c}{$\nongap{0}\gap{1}\gap{2}\nongap{3}\gap{4}\gap{5}\nongap{6}\gap{7}\gap{8}\nongap{9}\gap{10}\generator{11}\nongap{12}\generator{13}\nongap{14}\nongap{15}\dotscircles$} \\\end{tabular}} ][.{\begin{tabular}{c}{$\nongap{0}\gap{1}\gap{2}\nongap{3}\gap{4}\gap{5}\nongap{6}\gap{7}\gap{8}\nongap{9}\nongap{10}\gap{11}\nongap{12}\nongap{13}\generator{14}\nongap{15}\dotscircles$} \\\end{tabular}} ][.{\begin{tabular}{c}{$\nongap{0}\gap{1}\gap{2}\gap{3}\gap{4}\nongap{5}\nongap{6}\gap{7}\gap{8}\nongap{9}\nongap{10}\nongap{11}\nongap{12}\gap{13}\nongap{14}\nongap{15}\dotscircles$} \\\end{tabular}} ]][.{\begin{tabular}{c}{$\nongap{0}\gap{1}\gap{2}\gap{3}\gap{4}\gap{5}\gap{6}\nongap{7}\gap{8}\generator{9}\generator{10}\generator{11}\generator{12}\generator{13}\nongap{14}\generator{15}\dotscircles$} \\\end{tabular}} [.{\begin{tabular}{c}{$\nongap{0}\gap{1}\gap{2}\gap{3}\gap{4}\nongap{5}\gap{6}\nongap{7}\gap{8}\gap{9}\nongap{10}\generator{11}\nongap{12}\generator{13}\nongap{14}\nongap{15}\dotscircles$} \\\end{tabular}} ][.{\begin{tabular}{c}{$\nongap{0}\gap{1}\gap{2}\gap{3}\gap{4}\gap{5}\nongap{6}\nongap{7}\gap{8}\gap{9}\generator{10}\generator{11}\nongap{12}\nongap{13}\nongap{14}\generator{15}\dotscircles$} \\\end{tabular}} ][.{\begin{tabular}{c}{$\nongap{0}\gap{1}\gap{2}\gap{3}\gap{4}\gap{5}\nongap{6}\nongap{7}\gap{8}\nongap{9}\gap{10}\generator{11}\nongap{12}\nongap{13}\nongap{14}\nongap{15}\dotscircles$} \\\end{tabular}} ][.{\begin{tabular}{c}{$\nongap{0}\gap{1}\gap{2}\gap{3}\gap{4}\nongap{5}\gap{6}\nongap{7}\gap{8}\nongap{9}\nongap{10}\gap{11}\nongap{12}\generator{13}\nongap{14}\nongap{15}\dotscircles$} \\\end{tabular}} ][.{\begin{tabular}{c}{$\nongap{0}\gap{1}\gap{2}\gap{3}\gap{4}\gap{5}\nongap{6}\nongap{7}\gap{8}\nongap{9}\nongap{10}\gap{11}\nongap{12}\nongap{13}\nongap{14}\nongap{15}\dotscircles$} \\\end{tabular}} ][.{\begin{tabular}{c}{$\nongap{0}\gap{1}\gap{2}\gap{3}\gap{4}\nongap{5}\gap{6}\nongap{7}\gap{8}\nongap{9}\nongap{10}\nongap{11}\nongap{12}\gap{13}\nongap{14}\nongap{15}\dotscircles$} \\\end{tabular}} ]][.{\begin{tabular}{c}{$\nongap{0}\gap{1}\gap{2}\gap{3}\nongap{4}\gap{5}\gap{6}\gap{7}\nongap{8}\gap{9}\generator{10}\generator{11}\nongap{12}\generator{13}\nongap{14}\nongap{15}\dotscircles$} \\\end{tabular}} ][.{\begin{tabular}{c}{$\nongap{0}\gap{1}\gap{2}\gap{3}\gap{4}\nongap{5}\gap{6}\gap{7}\nongap{8}\gap{9}\nongap{10}\generator{11}\generator{12}\nongap{13}\generator{14}\nongap{15}\dotscircles$} \\\end{tabular}} ][.{\begin{tabular}{c}{$\nongap{0}\gap{1}\gap{2}\gap{3}\gap{4}\gap{5}\nongap{6}\gap{7}\nongap{8}\gap{9}\generator{10}\generator{11}\nongap{12}\generator{13}\nongap{14}\generator{15}\dotscircles$} \\\end{tabular}} [.{\begin{tabular}{c}{$\nongap{0}\gap{1}\gap{2}\gap{3}\nongap{4}\gap{5}\nongap{6}\gap{7}\nongap{8}\gap{9}\nongap{10}\gap{11}\nongap{12}\generator{13}\nongap{14}\generator{15}\dotscircles$} \\\end{tabular}} [.{\begin{tabular}{c}{$\nongap{0}\gap{1}\nongap{2}\gap{3}\nongap{4}\gap{5}\nongap{6}\gap{7}\nongap{8}\gap{9}\nongap{10}\gap{11}\nongap{12}\gap{13}\nongap{14}\generator{15}\dotscircles$} \\\end{tabular}} ]][.{\begin{tabular}{c}{$\nongap{0}\gap{1}\gap{2}\gap{3}\nongap{4}\gap{5}\nongap{6}\gap{7}\nongap{8}\gap{9}\nongap{10}\nongap{11}\nongap{12}\gap{13}\nongap{14}\nongap{15}\dotscircles$} \\\end{tabular}} ]][.{\begin{tabular}{c}{$\nongap{0}\gap{1}\gap{2}\gap{3}\gap{4}\gap{5}\gap{6}\nongap{7}\nongap{8}\gap{9}\generator{10}\generator{11}\generator{12}\generator{13}\nongap{14}\nongap{15}\dotscircles$} \\\end{tabular}} [.{\begin{tabular}{c}{$\nongap{0}\gap{1}\gap{2}\gap{3}\nongap{4}\gap{5}\gap{6}\nongap{7}\nongap{8}\gap{9}\gap{10}\nongap{11}\nongap{12}\generator{13}\nongap{14}\nongap{15}\dotscircles$} \\\end{tabular}} ][.{\begin{tabular}{c}{$\nongap{0}\gap{1}\gap{2}\gap{3}\gap{4}\gap{5}\nongap{6}\nongap{7}\nongap{8}\gap{9}\gap{10}\generator{11}\nongap{12}\nongap{13}\nongap{14}\nongap{15}\dotscircles$} \\\end{tabular}} ][.{\begin{tabular}{c}{$\nongap{0}\gap{1}\gap{2}\gap{3}\gap{4}\nongap{5}\gap{6}\nongap{7}\nongap{8}\gap{9}\nongap{10}\gap{11}\nongap{12}\nongap{13}\nongap{14}\nongap{15}\dotscircles$} \\\end{tabular}} ][.{\begin{tabular}{c}{$\nongap{0}\gap{1}\gap{2}\gap{3}\gap{4}\gap{5}\nongap{6}\nongap{7}\nongap{8}\gap{9}\nongap{10}\gap{11}\nongap{12}\nongap{13}\nongap{14}\nongap{15}\dotscircles$} \\\end{tabular}} ][.{\begin{tabular}{c}{$\nongap{0}\gap{1}\gap{2}\gap{3}\nongap{4}\gap{5}\gap{6}\nongap{7}\nongap{8}\gap{9}\nongap{10}\nongap{11}\nongap{12}\gap{13}\nongap{14}\nongap{15}\dotscircles$} \\\end{tabular}} ]][.{\begin{tabular}{c}{$\nongap{0}\gap{1}\gap{2}\gap{3}\nongap{4}\gap{5}\gap{6}\gap{7}\nongap{8}\nongap{9}\gap{10}\generator{11}\nongap{12}\nongap{13}\generator{14}\nongap{15}\dotscircles$} \\\end{tabular}} ][.{\begin{tabular}{c}{$\nongap{0}\gap{1}\gap{2}\gap{3}\gap{4}\gap{5}\nongap{6}\gap{7}\nongap{8}\nongap{9}\gap{10}\generator{11}\nongap{12}\generator{13}\nongap{14}\nongap{15}\dotscircles$} \\\end{tabular}} [.{\begin{tabular}{c}{$\nongap{0}\gap{1}\gap{2}\nongap{3}\gap{4}\gap{5}\nongap{6}\gap{7}\nongap{8}\nongap{9}\gap{10}\nongap{11}\nongap{12}\gap{13}\nongap{14}\nongap{15}\dotscircles$} \\\end{tabular}} ]][.{\begin{tabular}{c}{$\nongap{0}\gap{1}\gap{2}\gap{3}\gap{4}\gap{5}\gap{6}\nongap{7}\nongap{8}\nongap{9}\gap{10}\generator{11}\generator{12}\generator{13}\nongap{14}\nongap{15}\dotscircles$} \\\end{tabular}} [.{\begin{tabular}{c}{$\nongap{0}\gap{1}\gap{2}\gap{3}\gap{4}\gap{5}\nongap{6}\nongap{7}\nongap{8}\nongap{9}\gap{10}\gap{11}\nongap{12}\nongap{13}\nongap{14}\nongap{15}\dotscircles$} \\\end{tabular}} ]][.{\begin{tabular}{c}{$\nongap{0}\gap{1}\gap{2}\gap{3}\nongap{4}\gap{5}\gap{6}\gap{7}\nongap{8}\nongap{9}\nongap{10}\gap{11}\nongap{12}\nongap{13}\nongap{14}\generator{15}\dotscircles$} \\\end{tabular}} ][.{\begin{tabular}{c}{$\nongap{0}\gap{1}\gap{2}\gap{3}\gap{4}\nongap{5}\gap{6}\gap{7}\nongap{8}\nongap{9}\nongap{10}\gap{11}\generator{12}\nongap{13}\nongap{14}\nongap{15}\dotscircles$} \\\end{tabular}} ][.{\begin{tabular}{c}{$\nongap{0}\gap{1}\gap{2}\gap{3}\gap{4}\gap{5}\nongap{6}\gap{7}\nongap{8}\nongap{9}\nongap{10}\gap{11}\nongap{12}\generator{13}\nongap{14}\nongap{15}\dotscircles$} \\\end{tabular}} ][.{\begin{tabular}{c}{$\nongap{0}\gap{1}\gap{2}\gap{3}\gap{4}\gap{5}\gap{6}\nongap{7}\nongap{8}\nongap{9}\nongap{10}\gap{11}\generator{12}\generator{13}\nongap{14}\nongap{15}\dotscircles$} \\\end{tabular}} ][.{\begin{tabular}{c}{$\nongap{0}\gap{1}\gap{2}\gap{3}\gap{4}\nongap{5}\gap{6}\gap{7}\nongap{8}\nongap{9}\nongap{10}\nongap{11}\gap{12}\nongap{13}\nongap{14}\nongap{15}\dotscircles$} \\\end{tabular}} ][.{\begin{tabular}{c}{$\nongap{0}\gap{1}\gap{2}\gap{3}\gap{4}\gap{5}\gap{6}\nongap{7}\nongap{8}\nongap{9}\nongap{10}\nongap{11}\gap{12}\generator{13}\nongap{14}\nongap{15}\dotscircles$} \\\end{tabular}} ][.{\begin{tabular}{c}{$\nongap{0}\gap{1}\gap{2}\gap{3}\gap{4}\gap{5}\nongap{6}\gap{7}\nongap{8}\nongap{9}\nongap{10}\nongap{11}\nongap{12}\gap{13}\nongap{14}\nongap{15}\dotscircles$} \\\end{tabular}} ][.{\begin{tabular}{c}{$\nongap{0}\gap{1}\gap{2}\gap{3}\gap{4}\gap{5}\gap{6}\nongap{7}\nongap{8}\nongap{9}\nongap{10}\nongap{11}\nongap{12}\gap{13}\nongap{14}\nongap{15}\dotscircles$} \\\end{tabular}} ]] }  }; \node[align = center, below =.001\textwidth of arbre]
{\resizebox{.2\textwidth}{!}{{\ }}};\end{tikzpicture}}
\caption{$\T_7$, the ordinarization tree of numerical semigroups of genus $7$. This figure was created using the website of Bras-Amor\'os for drawing trees of numerical semigroups \cite{BA_draw}.}
\end{figure}
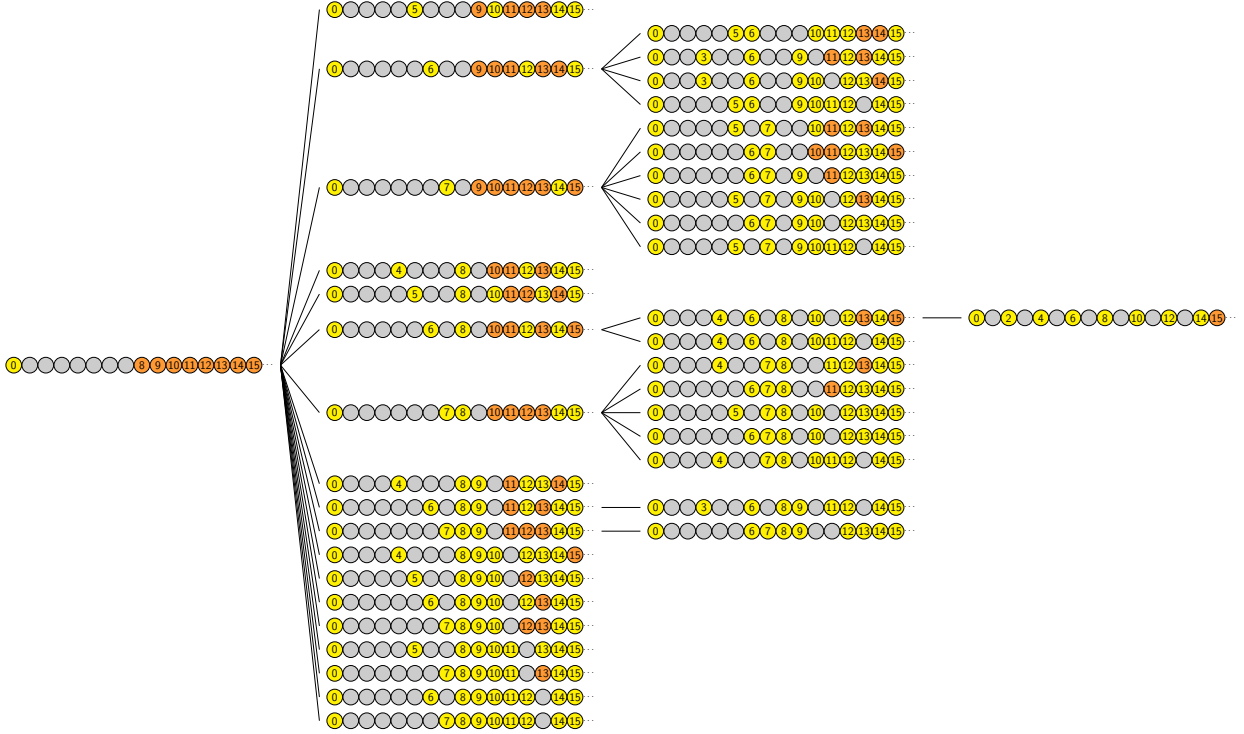

The \emph{ordinarization number} of a numerical semigroup $S$ of genus $g$ is the length of the path from $S$ back to $S_g$ in the ordinarization tree $\T_g$.  We denote this by $r(S)$.  For example, in Figure \ref{fig_T7} we see that there is a unique semigroup of genus $7$ with ordinarization number $3$, and that other than this one and the ordinary semigroup $S_7$, every other semigroup of genus $7$ has $r(S) \in \{1,2\}$.  Thinking about how the ordinarization transform changes a numerical semigroup, it is not difficult to prove the following characterization of $r(S)$.
\begin{prop}\cite[Lemma 1]{Bras-Amoros_ord}\label{Prop_count_small}
Let $S$ be a numerical semigroup of genus $g$.  Then $r(S) = \#\{S \cap \{1,2,\ldots, g\}\}$.
\end{prop}

We now state the conjecture that serves as partial motivation for this paper.  Let $n_{g,r}$ denote the number of numerical semigroups of genus $g$ with ordinarization number $r$. 
\begin{conj}[Bras-Amoros]\label{nrg_conj}\cite[Conjecture 7]{Bras-Amoros_ord}
For all $r, g \ge 1$ we have $n_{g,r} \le n_{g+1,r}$.
\end{conj}
\noindent Taking a sum over $r$, it is clear that Conjecture \ref{nrg_conj} implies Conjecture \ref{Ng_inc}.

Bras-Amor\'os proves this conjecture when $r = 1$ by giving a formula for $n_{g,1}$ \cite[Corollary 9]{Bras-Amoros_ord}.  In Section \ref{sec:ng2} we prove that Conjecture \ref{nrg_conj} holds for $r=2$ by giving a formula for $n_{g,2}$.  The strategy is similar to the proof for $r=1$, but the details are much more complicated.   Semigroups with genus $g$ and ordinarization number $2$ are in bijection with a certain collection of integer points in $\R_{\ge 0}^4$.  We use ideas from Ehrhart theory to count integer points in the relevant collection of polytopes and apply inclusion-exclusion to arrive at a final formula.  We give a formula for $n_{g,2}$ that is a quasipolynomial of degree $4$ and period $12$.  Once we have this formula, it is easy to check that $n_{g,2} \le n_{g+1,2}$ for all $g\ge 1$.  We study $n_{g,r}$ where $r$ is an arbitrary fixed positive integer and use these ideas to prove a result about the behavior of this function of $g$.

There is an additional result of Bras-Amor\'os that motivates the results of the rest of this paper.  
\begin{proposition} \cite[Lemma 2 and Lemma 4]{Bras-Amoros_ord}\label{Prop_BA_m2}
\begin{enumerate}
\item Let $S$ be a numerical semigroup of genus $g$.  Then $r(S) \le \lfloor\frac{g}{2}\rfloor$.
\item Let $g \ge 1$.  The unique numerical semigroup of genus $g$ and ordinarization number $\lfloor \frac{g}{2} \rfloor$ is $\{0,2,4,\ldots, 2g,2g+1, 2g+2,\ldots\}$.
\end{enumerate}
\end{proposition}
Bras-Amor\'os shows that the semigroup in the second part of this result is extremal with respect to the ordinarization number.  One of our major goals in this paper is to compute the ordinarization number for similar families of numerical semigroups.

Let $n_1, n_2, \ldots, n_t$ be positive integers.  The numerical semigroup generated by $n_1,\ldots, n_t$ is the set of all linear combinations of $n_1,\ldots, n_t$ with nonnegative coefficients.  It is denoted~by
\[
\langle n_1, n_2,\ldots, n_t \rangle = \{a_1 n_1 + \cdots + a_t n_t \colon a_1,\ldots, a_t \in \N_0\}.
\]
With this notation, the semigroup $S = \{0,2,4,\ldots, 2g,2g+1, 2g+2,\ldots\}$ is generated by $2$ and $2g+1$, that is, $S = \langle 2,2g+1\rangle$.  It is not difficult to show that every numerical semigroup $S$ has a unique minimal generating set.  The size of this minimal generating set is called the \emph{embedding dimension} of $S$, and is denoted by $e(S)$.  In Section \ref{sec:ord_e2} we consider ordinarization numbers of numerical semigroups with embedding dimension $2$ more generally.  In Section \ref{sec:ord_super}  we consider how these ideas generalize to supersymmetric numerical semigroups of larger embedding dimension.   In Section \ref{sec:interval} we consider ordinarization numbers of numerical semigroups generated by an interval of positive integers.

When we compute ordinarization numbers of families of numerical semigroups, we apply some results about the structure of the ordinarization tree $\T_g$.  In Section \ref{sec:ord_tree} we consider the degrees of vertices in $\T_g$.

\section{Degrees of Vertices in the Ordinarization Tree}\label{sec:ord_tree}

Let $S$ be a numerical semigroup of genus $g$.  We say that $S'$ is a \emph{child} of $S$ in $\T_g$ if $S$ is the ordinarization transform of $S'$, that is, $S' \cup \{F(S')\} \setminus \{m(S')\} = S$.  The main focus of this section is the study of the number of children of $S$ in $\T_g$.  For example, the semigroups of ordinarization number $1$ are precisely the children of the ordinary semigroup $S_g$.  Bras-Amor\'os has computed the number of these semigroups.
\begin{prop}\cite[Lemma 8]{Bras-Amoros_ord}\label{Ord1}
Let $g$ be a nonnegative integer.  The number of semigroups of genus $g$ and ordinarization number $1$ is 
\[
n_{g,1} = 
\begin{cases}
\frac{3}{8} g^2 - \frac{1}{4} g & \text{ if } g \text{ is even,}\\
\frac{3}{8} g^2 - \frac{3}{8} & \text{ if } g \text{ is odd.}
\end{cases}
\]
\end{prop}

Considering the trees $\T_g$ for small values of $g$, one notices that there are many numerical semigroups that have no children in $\T_g$.  Throughout the rest of this section, we let $\mathcal{S}_g$ denote the set of numerical semigroups of genus $g$.
\begin{question}
Let $s$ be a nonnegative integer.  As $g$ goes to infinity, what proportion of semigroups of genus $g$ have exactly $s$ children in the ordinarization tree $\T_g$?  That is, what is
\[
\lim_{g\rightarrow \infty} \frac{\#\{S \in \mathcal{S}_g \colon S \text{ has } s \text{ children in } \T_g\}}{N(g)}? 
\]
\end{question}
Our motivation for asking this question comes from the study of the \emph{semigroup tree}.  The semigroup tree is a rooted tree with root $\N_0$ where the vertices at level $g$ are exactly the numerical semigroups of genus $g$.  The easiest way to describe the tree is by explaining how to find the path from a numerical semigroup $S$ back up to $\N_0$.  If $S$ has genus $g$, then $S' = S \cup \{F(S)\}$ is a numerical semigroup of genus $g-1$.  Repeating this process, we get a numerical semigroup of genus $g-2$, followed by one of genus $g-3$, all the way back up to $\N_0$, the unique numerical semigroup of genus $0$.  

This description does not directly explain how to produce level $g+1$ of the tree given the first $g$ levels of the tree.  Let $S$ be a numerical semigroup.  We say that $S'$ is a child of $S$ in the semigroup tree if $S = S' \cup \{F(S')\}$.  We now explain how to describe the set of children of $S$.  Minimal generators of $S$ larger than $F(S)$ are called \emph{effective generators}, or sometimes \emph{right generators}, of $S$.  Let $\eg(S)$ denote the set of effective generators of $S$.  The number of effective generators of $S,\ |\eg(S)|$, is sometimes called the \emph{effectivity} of $S$ and is denoted by $h(S)$.  It is not so difficult to show that the children of $S$ in the semigroup tree are the semigroups of the form $S' = S \setminus \{n\}$ where $n \in \eg(S)$.  Therefore, the number of children of $S$ in the semigroup tree is $h(S)$.

One motivation for studying effective generators of numerical semigroups comes from Conjecture \ref{Ng_inc}.  If we want to show that $N(g+1) \ge N(g)$, then all we have to do is show that for every $g\ge 1$ we have
\[
\frac{\sum_{S\in \mathcal{S}_g} h(S)}{N(g)} \ge 1.
\]
In order to show that this inequality holds, we need only show that there are not `too many' numerical semigroups $S \in \mathcal{S}_g$ with $h(S) = 0$.  O'Dorney determined the proportion of numerical semigroups of genus $g$ that have a given number of effective generators.
\begin{thm}\cite[Theorem 2]{ODorney}
Let $t(g,h) = \#\{S \in \mathcal{S}_g \colon h(S) = h\}$ and let $C$ be the constant from Theorem \ref{thm:Zhai}.  Then
\[
\sum_{h \ge 0} |t(g,h) - C \varphi^{g-(h+2)}| = o(\varphi^g).
\]
\end{thm}
In particular, the proportion of numerical semigroups of genus $g$ with $h(S) = 0$ approaches $\varphi^{-2}$ as $g$ goes to infinity.  Note that since $\sum_{h \ge 0} h \varphi^{-(h+2)} = \varphi$, this result is consistent with the fact that Theorem \ref{thm:Zhai} implies
\[
\lim_{g \rightarrow \infty} \frac{\sum_{S\in \mathcal{S}_g} h(S)}{N(g)} = \varphi.
\]

One of our primary motivations for studying degrees of vertices in $\T_g$ is the following question.
\begin{question}
Is there a way to partition the set of vertices in $\T_g$ into sets such that the average value of $h(S)$ within each set is at least one?
\end{question}

We next note the role played by effective generators in the structure of the ordinarization tree~$\T_g$.
\begin{theorem}\label{strictly}
Suppose that $S'$ is a child of $S$ in $\T_g$.  Then,
\begin{enumerate}
\item $\eg(S') \subseteq \eg(S)$, and 
\item $F(S')$ is an element of $\eg(S) \setminus \eg(S')$.
\end{enumerate}
Therefore, $h(S) > h(S')$.
\end{theorem}

Before giving the proof, we recall a fact about the relationship between the minimal generators of the numerical semigroups $S$ and $T = S \cup \{F(S)\}$.  This result follows from~\cite[Lemma 3]{Bras-Amoros_towards}.
\begin{lemma}\label{Lem:min_gens}
Let $S$ be a numerical semigroup.  Suppose that $S \neq S_g$ and $s$ is a minimal generator of $S$.  If $s \neq m(S)+F(S)$, then $s$ is a minimal generator of the numerical semigroup $T = S \cup \{F(S)\}$.
\end{lemma}

\begin{proof}[Proof of Theorem \ref{strictly}]
Suppose $n \in \eg(S')$.  Since $F(S') > F(S)$, we see that $n > F(S)$.  In order to prove that $\eg(S') \subseteq \eg(S)$ we need only show that $n$ is a minimal generator of $S$.

Write $S' = S \setminus \{a\} \cup \{b\}$ for some element $a \in S$ and $b \in \N_0 \setminus S$.  Since $S'$ is a child of $S$ in $\T_g$, we see that $F(S') = a$ and $m(S') = b$.  Consider the numerical semigroup $T = S' \cup \{a\} = S \cup \{b\}$.  Suppose that $n \neq F(S')+m(S') = a+b$.  Lemma \ref{Lem:min_gens} implies that $n$ is a minimal generator of $T$, which implies that $n$ is a minimal generator of $S = T\setminus \{b\}$.

Now suppose $n = F(S')+m(S')$ and that $n$ is not a minimal generator of $S$.  This implies that $n$ is a nonnegative linear combination of elements of $S$ smaller than $n$.  Since $n$ is a minimal generator of $S'$, it is not possible to write $n$ as a nonnegative linear combination of elements of $S'$ smaller than $n$.  Since $F(S')$ is the only element of $S$ that is not an element of $S'$, we see that $n$ equals $F(S')$ plus some nonnegative linear combination of elements of $S$.  Since $m(S') < m(S)$, this is not possible, which gives a contradiction.

For the second part of the statement, we note that $F(S')$ must be a minimal generator of $T$, which implies that it is a minimal generator of $S$.
\end{proof}

\begin{cor}
If $S \in \mathcal{S}_g$ has $h(S) = 0$, then $S$ has no children in $\T_g$.
\end{cor}

We now prove a result showing that it is not possible for a semigroup to have too many children in $\T_g$ that have no effective generators.
\begin{theorem}\label{Thm:children_h0}
The number of children of $S$ in $\T_g$ that have no effective generators is at most $\big\lfloor \frac{m(S)}{2}\big\rfloor$.
\end{theorem}

We apply the following lemma in the proof of this theorem.
\begin{lemma}\label{ineffectiveGaps}
Suppose that $S'$ is a child of $S$ in $\T_g$.  Write $S' = S \setminus \{a\} \cup \{b\}$ for some element $a \in S$ and $b \in \N_0 \setminus S$.  
We have $a \in \eg(S)$ and  $\lceil \frac{m(S)}{2}\rceil \le b \le m(S)-1$.
\end{lemma}

\begin{proof}
Since $S'$ is a child of $S$ in $\T_g$ we have $m(S') = b$, so $b \le m(S) - 1$.  Since $S'$ contains exactly one element of $\N_0 \setminus S$ less than $m(S)$, we must have $2b \ge m(S)$.  

Since $S'$ is a child of $S$ in $\T_g$ we have $F(S') = a$, which implies $a > F(S)$.  We see that $a$ must be a minimal generator of $S$. Otherwise $S\setminus \{a\}$ is not closed under addition, which would imply $S' = S\setminus \{a\} \cup\{b\}$ is also not closed under addition.
\end{proof}

\begin{proof}[Proof of Theorem \ref{Thm:children_h0}]
We show that for each $b$ satsifying $\lceil \frac{m(S)}{2}\rceil \le b \le m(S)-1$, there is at most one child $S'$ of $S$ in $\T_g$ for which $m(S') = b$ where $S'$ has no effective generators.  Let $\eg(S) = \{n_1, \ldots, n_h\}$.

Suppose $S'$ and $S''$ are two children of $S$ in $\T_g$ and $m(S') = m(S'') = b$.  Then $S' = S \setminus \{n_i\} \cup \{b\}$ and $S'' = S \setminus \{n_j\} \cup \{b\}$.  Lemma \ref{ineffectiveGaps} implies $\lceil \frac{m(S)}{2}\rceil \le b \le m(S)-1$.  Without loss of generality, assume $n_i < n_j$.  Note that $F(S') = n_i < n_j$.  We complete the proof by showing that $n_j$ is an effective generator of $S'$.

Consider the semigroup $T = S' \cup \{F(S')\} = S'' \cup\{F(S'')\} = S \cup \{b\}$.  Since $T \setminus \{n_i\}$ and $T \setminus \{n_j\}$ are numerical semigroups, we see that $n_i$ and $n_j$ are both minimal generators of $T$.

We now show that $n_j$ is an effective generator of $S'$.  Since $F(S') = n_i < n_j$, we need only show that $n_j$ is a minimal generator of $S'$.  Suppose that it is not.  Then $n_j$ is a linear combination with nonnegative coefficients of elements of $S'$ that are less than $n_j$.  Since $S''$ contains every element of $S'$ less than $n_j$, we see that $n_j$ cannot be a minimal generator of $T = S \cup\{b\}$, which is a contradiction.
\end{proof}
We give a simple example to show that this result is sharp.  Let $S = \langle 7,8,10,11,12,13\rangle$.  One can check that $S$ has $3$ children in $\T_7$ that have no effective generators.

Examining the trees $\T_g$ for small values of $g$, it seems that every semigroup $S$ with at least one child in $\T_g$ has a large number of effective generators relative to $m(S)$.  We show that this does not necessarily hold for large genus. 
\begin{prop}\label{prop:h4_parent}
There exist infinitely many numerical semigroups $S$ such that $h(S) = 4$ and $S$ has at least one child in the ordinarization tree.
\end{prop}

In order to describe the family of numerical semigroups used to prove this proposition, we introduce some additional notation.  The \emph{Ap\'ery set} of $S$ with respect to its multiplicity $m(S) = m$ is $\Ap(S;m) = \{n \in S \colon n -m \not\in S\}$.  It is not difficult to see that $\Ap(S;m)$ consists of one element in each residue class modulo $m(S)$.  Let 
\[
\Ap(S;m) = \{0,k_1 m + 1,\ldots, k_{m-1}m + m-1\}.
\]  
Since $m$ is the multiplicity of $S$, each $k_i \ge 1$.  The Ap\'ery set of a numerical semigroup with respect to its multiplicity $m$ determines the set of gaps in each residue class modulo $m$.  This implies that $F(S) = \max \Ap(S;m) - m$ \cite[Proposition 2.12]{GarciaSanchez_Rosales}.  It is not difficult to see that an element of $\Ap(S;m)$ is a minimal generator of $S$ if and only if it is not the sum of two other elements of $\Ap(S;m)$.  The \emph{Kunz coordinate vector},  or \emph{Ap\'ery tuple}, of $S$ is $(k_1,\ldots, k_{m-1})$.  Not every vector of $m-1$ nonnegative integers is the Kunz coordinate vector of a numerical semigroup of mutliplicity $m$.  For a characterization of the tuples $(k_1,\ldots, k_{m-1})$ that occur as Kunz coordinate vectors, see \cite{kunz, kunzcoords}.

\begin{proof}[Proof of Proposition \ref{prop:h4_parent}]
For any positive integer $k$ there is a numerical semigroup $T_k$ of multiplicity $2k+1$ with Kunz coordinate vector $(k,k,k-1,k-1,\ldots,2,2,1,1)$.  It is an exercise to check that $T_k' = T_k \cup \{2k\}$ is closed under addition.  We can verify that $F(T_k) = (k-1) (2k+1) +2$ and when $k \ge 2$ we have 
\[
\eg(T_k) = \{(2k+1)k+1, (2k+1)k+2,(2k+1)(k-1)+3,(2k+1)(k-1)+4\},
\] 
and therefore $h(T_k) = 4$. It is not difficult to check that $T_k'' = T_k \cup \{2k\} \setminus \{k(2k+1) +1\}$ is a child of $T_k$ in the ordinarization tree.
\end{proof}

\section{Semigroups with Fixed Ordinarization Number}\label{sec:ng2}

The main motivation for this section comes from Conjecture \ref{nrg_conj}.
\begin{question}
Let $r$ be a fixed positive integer.  How does $n_{g,r}$ behave as a function of $g$?
\end{question}

Proposition \ref{Ord1} answers this question when $r=1$.  There is one quadratic polynomial that gives the formula for $n_{g,1}$ when $g$ is even, and a different quadratic polynomial that gives the formula for $n_{g,1}$ when $g$ is odd.  We show that this kind of behavior holds for any value of $r$.
\begin{defn}
A \emph{quasipolynomial} $Q$ is an expression of the form
\[
Q(t) = c_n(t) t^n + \cdots + c_1(t) t + c_0(t)
\]
where $c_0,\ldots, c_n$ are periodic functions in $t$.  Assume that $c_n$ is not the zero function.  The \emph{degree} of $Q$ is $n$, and the least common period of $c_0,\ldots, c_n$ is the \emph{period} of $Q$.

A function $f\colon \N \rightarrow \N$ is \emph{eventually quasipolynomial} if there exists a quasipolynomial $Q$ such that $f(t) = Q(t)$ for all sufficiently large $t$.  
\end{defn}

\begin{example}
The formula for $n_{g,1}$ from Proposition \ref{Ord1} is a quasipolynomial of degree $2$ and period $2$.  We see that $c_2$ is the constant function $\frac{3}{8}$ and each of $c_1$ and $c_0$ have period~$2$.
\end{example}
The main result of this section is to show that we get a formula similar to this one for any fixed value of $r$.

\begin{thm}\label{thm:quasi_ord}
Fix a positive integer $r$.  There is a formula for $n_{g,r}$ that is eventually quasipolynomial in $g$ of degree $2r$.
\end{thm}
Before giving the proof of this theorem we focus on the case $r=2$.  The semigroups of genus $g$ and ordinarization number $2$ are the semigroups that are the children of the children, the `grandchildren', of the ordinary numerical semigroup $S_g = \{0,g+1,g+2,\ldots \}$. 

\begin{thm}\label{Thm:ng2}
There is a formula for $n_{g,2}$ that is a quasipolynomial in $g$ of degree $4$ and period $12$.  That is, there exist polynomials $f_0(x),\ldots, f_{11}(x) \in \Q[x]$ such that $n_{g,2} = f_i(g)$ whenever $g$ is a positive integer satisfying $g \equiv i \pmod{12}$.  Each of these polynomials has leading coefficient equal to $\frac{11}{384}$, so for ease of notation we define $f_i'(x) = \frac{384}{11} f_i(x)$.  These polynomials are as follows: 
\begin{eqnarray*}
f_0'(x) & = &  x  (x^3 - 2548/297 x^2 + 336/11 x - 1376/33) \\
 f_1'(x) & = &   (x - 1) (x^3 - 1927/297 x^2 + 3611/297 x - 541/297)\\
  f_2'(x) & = & (x - 2)   (x^3 - 1954/297 x^2 + 5548/297 x - 3352/297) \\
    f_3'(x) &=& (x - 3)  (x^3 - 1333/297 x^2 + 449/99 x + 51/11)\\
  f_4'(x) &=&  (x^4 - 2548/297 x^3 + 3088/99 x^2 - 4192/99 x - 512/297)\\
 f_5'(x) &= & (x^4 - 2224/297 x^3 + 1910/99 x^2 - 1576/99 x - 6499/297)\\
  f_6'(x) &=& (x^4 - 2548/297 x^3 + 336/11 x^2 - 1520/33 x + 48)\\
    f_7'(x) &=& (x^4 - 2224/297 x^3 + 1846/99 x^2 - 952/99 x - 4643/297)\\
  f_8'(x) &= & (x^4 - 2548/297 x^3 + 3152/99 x^2 - 4384/99 x - 7552/297)\\
 f_9'(x) &= & (x^4 - 2224/297 x^3 + 18 x^2 - 40/3 x + 39/11)\\
  f_{10}'(x) &= &(x^4 - 2548/297 x^3 + 3088/99 x^2 - 4624/99 x + 13744/297)\\
    f_{11}'(x) &= & (x + 1)  (x^3 - 2521/297 x^2 + 8251/297 x - 11683/297).
\end{eqnarray*}

\end{thm}
Bras-Amor\'os gives the values for $n_{g,2}$ for $g \le 49$ in \cite{Bras-Amoros_ord}. We have checked that the formula given here is consistent with these computations.

It is not difficult to check that this quasipolynomial formula is increasing in $g$.
\begin{cor}
For all $g\ge 1$, we have $n_{g,2} \le n_{g+1,2}$.  That is, Conjecture \ref{nrg_conj} holds for $r=2$.
\end{cor}

\begin{remark}
In Theorem \ref{thm:quasi_ord} we show that for a fixed positive integer $r$, the sequence $n_{g,r}$ is \emph{eventually} quasipolynomial.  But, the formula for $r=1$ given in Proposition \ref{Ord1} and the formula for $r=2$ given in Theorem \ref{Thm:ng2}, are quasipolynomial.  That is, these formulas hold for all positive values of $g$, not only values of $g$ that are `sufficiently large'.  We do not know whether it is possible to replace `eventually quasipolynomial' with `quasipolynomial' in the statement of Theorem \ref{thm:quasi_ord}.
\end{remark}

We first recall a simple proposition about the gaps of a numerical semigroup of genus $g$.
\begin{prop}\label{prop:Fbound}\cite[Lemma 2.14]{GarciaSanchez_Rosales}
Let $S$ be a numerical semigroup.  Then $F(S) \le 2g(S)-~1$.
\end{prop}

We also recall some basic definitions from polyhedral geometry.
\begin{defn}
A \emph{rational polyhedron} $P \subset \R^d$ is the set of solutions to a finite list of linear inequalities with rational coefficients.  That is,
\[
P = \{x \in \R^d \colon A x \ge b\}
\]
for some matrix $A$ and vector $b$ with rational entries.  If $P$ is bounded, we say that it is a \emph{rational polytope}.  If $b$ is the zero vector, then the system defining $P$ is called \emph{homogeneous}. If $b$ is nonzero, this system is \emph{inhomogeneous}.

Let $P \subset \R^d$ be a rational polytope and $n\in \N$. Define 
\[
i(P,n) = \#\{nP\cap \Z^d\}
\]
where $nP = \{n \alpha\colon \alpha \in P\}$.  The function $i(P,n)$ is called the \emph{Ehrhart quasipolynomial} of $P$ and the function
\[
1+ \sum_{n=0}^\infty i(P,n) z^n
\]
is the \emph{Ehrhart series} of $P$.
\end{defn}
A famous theorem of Ehrhart, see for example \cite[Theorem 4.6.8]{StanleyEC1}, implies that if $P$ is a rational polytope, $i(P,n)$ is a quasipolynomial in $n$.

\begin{proof}[Proof of Theorem \ref{thm:quasi_ord}]
Suppose $S \in \Sg_g$ has ordinarization number $r$.  Then 
\[
S = S_g \setminus \{a_1,\ldots, a_r\} \cup \{b_1,\ldots, b_r\}
\] 
where $b_1,\ldots, b_r \le g$.  Without loss of generality we can assume that 
\[
1\le b_r < b_{r-1} < \cdots < b_1 \le g\ \ \ \text{ and }\ \ \  g+1 \le a_1 < a_2 < \cdots < a_r \le 2g-1,
\] 
where we have applied Proposition \ref{prop:Fbound} to see that $a_r \le 2g-1$.  We find the conditions on $(a_1,\ldots, a_r, b_1, \ldots, b_r) \in \Z_{\ge 1}^{2r}$ that determine whether $S$ is closed under addition.

First suppose that $x,y \in S \setminus \{b_1,\ldots, b_r\}$.  Then $x+y \ge 2(g+1) > a_r = F(S)$, so $x+y \in S$. 

Consider $b_i + b_j$ where $1\le i \le j \le r$.  If $b_i + b_j \le g$ we must have $b_i + b_j = b_k$ for some $k$ satisfying $1\le k < i$.   If $b_i + b_j \ge g+1$, then $b_i + b_j \in S$ if and only if $b_i + b_j \neq a_k$ for any $k \in \{1,2,\ldots, r\}$.

For each pair $(i,j)$ satisfying $1\le i \le j \le r$ we have a choice to make.  
\begin{itemize}
\item Either $b_i + b_j = b_k$ for some $k$ satisfying $1 \le k < i$, or 
\item $b_i + b_j = a_k$ for some $k \in \{1,2,\ldots, r\}$, or
\item neither of these situations occurs.
\end{itemize}
For each pair $(i,j)$ we choose whether $b_i + b_j = b_k$ for some $k$, or $b_i + b_j = a_k$ for some $k$, or neither of these equalities holds.  This collection of equalities together with the conditions that $1\le b_r < b_{r-1} < \cdots < b_1 \le g$ and $g+1 \le a_1 < a_2 < \cdots < a_r \le 2g-1$, defines a polytope $P(g) \subset \R_{\ge 0}^{2r}$.  In order to deal with strict inequalities like $b_i < b_{i-1}$ we can impose additional constraints of the form $b_i \le b_{i-1}+1$.  In this way, we can describe $P(g)$ without using any strict inequalities.

There is a finite collection of polytopes $P_1(g),\ldots, P_N(g)$ that arise in this analysis.  
Let $N_i(g)$ be the number of integer points in $P_i(g)$.  Applying inclusion-exclusion, we see that $n_{g,r}$ is equal to a linear combination with integer coefficients of the functions $N_i(g)$.

The function $N_j(g)$ is not given by counting integer points in the $g$\textsuperscript{th} dilate of the polytope $P_j(1)$, so we cannot directly apply the most basic version of Ehrhart's theorem.  However, there are standard tools in Ehrhart theory to deal with this kind of enumeration in the inhomogeneous case.  We first introduce a new variable corresponding to the value of $g$ so that for each $j$, the collection of polytopes $P_j(g)$ as $g$ varies can be considered in terms of a single rational polyhedron $P_j' \subset \R^{2r+1}$.  One can replace an arbitrary polyhedron $P \subset \R^d$ with the \emph{cone over $P$} in $\R^{d+1}$, which amounts to passing from an inhomogeneous system to a homogeneous one by introducing an additional homogenizing variable \cite[Section 4]{normaliz}.  We then want to count the number of integer points in this polyhedron with a fixed value of $g$.  We can do this by choosing a grading on $\Z^d$ given by the value of this coordinate corresponding to $g$.  A rational cone $C$ and a grading together define a rational polytope, and it is the Ehrhart series of this polytope that we want to compute \cite[Section A.6]{Normaliz2}.  We apply these ideas to the polytopes $P_1(g),\ldots, P_N(g)$.  A variation of Ehrhart's theorem attributed to Ehrhart, Stanley, and Hilbert-Serre in \cite[Theorem 15]{normaliz}, then implies that each $N_i(g)$ is eventually quasipolynomial of degree at most~$2r$.

We now show that the degree of the final quasipolynomial formula for $n_{g,r}$ that holds for all $g$ sufficiently large cannot be less than $2r$.  We do this by giving a lower bound for $n_{g,r}$ that is at least a constant times $g^{2r}$.  

Every choice of $(a_1,\ldots, a_r, b_1,\ldots, b_r)$ satisfying 
\[
\frac{2g}{3} \le b_r< b_{r-1}<\cdots < b_1 \le g <  a_1 < a_2 < \cdots < a_r \le \frac{4g}{3}-1
\] 
gives a numerical semigroup of ordinarization number $r$.  There is a constant $c$ such that for all $g$ sufficiently large, the number of these choices is at least $c g^{2r}$.
\end{proof}

We now show how to carry out the ideas described in the proof of this theorem for small values of $r$. We give a proof of Proposition \ref{Ord1} that is slightly different from the one in \cite{Bras-Amoros_ord}.  We then prove Theorem \ref{Thm:ng2}.  The computations in this section make extensive use of the computer algebra systems Sage \cite{SageMath} and Normaliz \cite{Normaliz_system}.
\begin{proof}[Proof of Proposition \ref{Ord1}]
Suppose $S$ is a child of the ordinary semigroup $S_g$ in $\T_g$.  Then $S = S_g \setminus \{a\} \cup \{b\}$ for some $b \in \{1,2,\ldots, g\}$ and $a \ge g+1$.  We count the children of $S_g$ by determining the conditions on $a$ and $b$ so that $S_g \setminus \{a\} \cup \{b\}$ is closed under addition, and counting the pairs $(a,b)$ satisfying these conditions. Proposition \ref{prop:Fbound} implies that $a \le  2g-1$ and Lemma \ref{ineffectiveGaps} implies that $2b \ge g+1$.  If $a - b \in S$, then $a-b + b  = a \in S$, which is a contradiction.  Therefore, we must have $a-b \le g$ and $a-b \neq b$. 

We check that if $(a,b)$ satisfies all of these conditions, then $S_g \setminus \{a\} \cup \{b\}$ is closed under addition.  This allows us to conclude that $S$ is a child of $S_g$ in $\T_g$.  
\begin{itemize}[leftmargin=*]
\item If $x\in S$ and $x > b$, then $x \ge g+1$.  If we add two elements of $S$ that are larger than $b$, their sum is at least $2g+2> F(S) = a$.  
\item By assumption $2b \neq a$. Since $2b \ge g+1$, we see that $2b \in S$.  
\item Suppose $x$ is a nonzero element of $S$ and $x \neq b$. Since $b+x \ge g+1$ and $a-b \not\in S$, we see that $b+x \neq a$.  We conclude that $b+x \in S$.
\end{itemize}

Let $P_1(g) \subseteq \R_{\ge 0}^2$ be the polytope defined by the inequalities: $g+1 \le a \le 2g-1,\ a-b \le g$, and $g \ge b \ge \frac{g+1}{2}$.  As described in the proof of Theorem \ref{thm:quasi_ord}, this family of polytopes can be studied in terms of the polyhedron $P' \subset \R^3$ where we introduce a new variable corresponding to $g$:
\[
P' = \left\{ \left(\begin{array}{c} 
a \\ b \\ g
\end{array}\right) \in \R_{\ge 0}^3 \colon
\left(\begin{array}{ccc} 
1 & 0 & -1\\
-1 & 0 & 2\\
0 & -1 & 1\\
-1 & 1 & 1\\
0 & 2 & -1
\end{array}\right)
\left(\begin{array}{c} 
a\\
b\\
g
\end{array}\right)
\ge 
\left(\begin{array}{c} 
1 \\
1 \\
0 \\
0 \\
1 
\end{array}\right)
\right\}.
\]
Let $P_2(g) \subseteq \R_{\ge 0}^2$ be the polytope defined by these inequalities together with the additional equality $2b = a$.  We write $N_1(g)$ for the number of integer points in $P_1(g)$ and $N_2(g)$ for the number of integer points in $P_2(g)$.  We see that $n_{g,1} = N_1(g) - N_2(g)$.  In order to verify the formula given in Proposition \ref{Ord1}, it is enough to show that $N_2(g) = \left\lfloor \frac{g-1}{2}\right\rfloor$, and 
\[
N_1(g) = 
\begin{cases} 
\frac{3}{8} (g+2) \left(g-\frac{4}{3}\right) & \text{if } g \equiv 0 \pmod{2} \\ 
\frac{3}{8} (g-1)\left(g+\frac{7}{3}\right) & \text{if } g \equiv 1 \pmod{2}
\end{cases}.
\]
We create the polyhedron $P'$ in the computer algebra system Sage and then count points in it with a given value of the coordinate corresponding to $g$.  We do this via the following commands:\\
\texttt{ord1poly = Polyhedron(ieqs=$[(-1,1,0,-1),(-1,-1,0,2),(0,0,-1,1),(0,-1,1,1),\\
(-1,0,2,-1)]$, backend = `normaliz')\\
hilb = ord1poly.hilbert\_series($[0,0,1]$)}.\\
We note that the vector $[0,0,1]$ specifies the grading based on the value of $g$.  

This command produces the Hilbert series 
 \[
\frac{t  (t^{2} + t + 1)}{ (1 + t)^{2}  (1 - t)^{3}}.
 \]
 From here, it is straightforward to deduce the formula for $N_1(g)$.  The analysis for $N_2(g)$ is similar.
\end{proof}

\begin{proof}[Proof of Theorem \ref{Thm:ng2}]
We follow the strategy of the proof of Proposition \ref{Ord1}.  Suppose $S \in \Sg_g$ has ordinarization number $2$.  Then $S = S_g \setminus \{a_1, a_2\} \cup \{b_1, b_2\}$ for some $1\le b_2 < b_1 \le g$ and $g+1 \le a_1 < a_2$.   Proposition \ref{prop:Fbound} implies that $a_2 \le 2g-1$.  We compute $n_{g,2}$ by determining the conditions on $(a_1,a_2,b_1,b_2)$ so that $S = S_g \setminus \{a_1, a_2\} \cup \{b_1, b_2\}$ is closed under addition, and counting the tuples $(a_1,a_2,b_1,b_2)$ satisfying these conditions.

Suppose $S$ is closed under addition.  Since $2b_2 \in S$, either $2b_2 = b_1$ or $2b_2 \ge g+1$.  None of $2b_2, b_1+b_2$, or $2b_1$ can be equal to $a_1$ or $a_2$.  Since $(a_2 - b_2) + b_2  = a_2$, we have $a_2 -b_2 \not\in S$.  Similarly $a_2-b_1, a_1-b_2, a_1-b_1 \not\in S$.  We see that $a_2 - b_2 \le g$ or $a_2 - b_2 = a_1$.  

We split the analysis into four cases based on whether $2b_2 = b_1$ and whether $a_2 - b_2 = a_1$.  We first consider that case where $2b_2 \neq b_1$ and $a_2-b_2 \neq a_1$.  

We check that if $(a_1,a_2,b_1,b_2) \in \Z_{\ge 1}^4$ satisfy the following conditions, then $S$ is closed under addition and is therefore a numerical semigroup of ordinarization number $2$:
\begin{eqnarray*}
1 \le b_2 < b_1 \le g, &  & g+1 \le a_1 < a_2 \le 2g-1, \\
2b_2 \ge g+1, \ & & 
a_2 - b_2 \le g, \\
b_i + b_j \neq a_k & & \text{ for any } i,j,k \in \{1,2\}.
\end{eqnarray*}

We have $2b_2 \ge g+1$, so we also have $2b_1 > b_1+b_2 > 2b_2 \ge g+1$.  By assumption $b_i+b_j \neq a_k$ for any $\{i,j,k\} \in \{1,2\}$. Since $a_1, a_2$ are the only gaps of $S$ of size at least $g+1$, we see that $2b_2, b_1 + b_2, 2b_1 \in S$.  Suppose $x \in S$ satisfies $x \ge g+1$.  Then $b_2 + x \in S$ since $b_2 + g+1 > a_2$.  Similarly, $b_1 + x \in S$.  The sum of two elements of $S$ of size at least $g+1$ has size at least $2(g+1)> a_2$.  We conclude that $S$ is closed under addition.  

In this first case we required that $2b_2 \ge g+1$ and $a_2 - b_2 \le g$.  If $2b_2 = b_1$, then the inequality $2b_2 \ge g+1$ does not hold.  We instead replace it with the inequality $b_1 + b_2 \ge g+1$.  If $a_2 - b_2 = a_1$, then the inequality $a_2 - b_2 \le g$ does not hold.  We instead replace it with the pair of inequalities $a_2 - b_1 \le g$ and $a_1 - b_2 \le g$.  The three remaining cases are similar and we omit the details.

For each positive integer $g$, the conditions 
\begin{eqnarray*}
1 \le b_2 < b_1  \le g, & &g+1 \le a_1 < a_2 \le 2g-1 \\
2b_2 \ge g+1, & & a_2 - b_2 \le g
\end{eqnarray*}
define a polytope $P^*(g) \subseteq \R_{\ge 0}^4$.  Let $N^*(g)$ be the number of integer points in $P^*(g)$.  It is not the case that $n_{g,2}$ is equal to $N^*(g)$ because we have not accounted for possibilities $(a_1,a_2,b_1,b_2)$ satisfying at least one of the following equalities: $2b_2 = b_1,\ a_2 - b_2 = a_1$, and $b_i+b_j = a_k$ for some $\{i,j,k\} \in \{1,2\}$.  

We proceed by applying inclusion-exclusion.  We specify a subset of these inequalities to hold and count points in the resulting polytope.  We combine these formulas at the end to obtain an expression for $n_{g,2}$.  The analysis is not as complicated as it might seem because many of these collections of equalities cannot occur together.  For example, if $2b_2 = a_2$, then $2b_1 > b_1 + b_2 > 2b_2 = a_2$, and so $2b_1$ and $b_1 + b_2$ cannot be equal to $a_1$ or $a_2$.

As described in the proof of Theorem \ref{thm:quasi_ord}, the family of polytopes $P^*(g)$ can be studied in terms of a polyhedron $P' \subset \R^5$ where we introduce a new variable corresponding to $g$.  We can use Sage to compute the Hilbert series of $P'$ as follows:\\
\texttt{\small{Pstar = Polyhedron(ieqs=$[(0,0,0,-1,0,1),(-1,0,0,1,-1,0),(-1,1,0,0,0,-1),\\
(-1,-1,1,0,0,0),(-1,0,-1,0,0,2),(0,0,-1,0,1,1),(-1,0,0,0,2,-1)]$, backend = `normaliz')\\
hilb = Pstar.hilbert\_series($[0,0,0,0,1]$)}}.\\
This returns the Hilbert series 
\[
-(t - 1)^{-5}  (t + 1)^{-4}  t^{3}  (t^{4} + 2t^{3} + 5t^{2} + 2t + 1),
\]
from which we can derive a quasipolynomial expression that holds for all positive integers $g$:
\[
N^*(g) = \begin{cases}
\frac{11}{384} (g - 2)  g  \left(g^2 - \frac{6}{11}g - \frac{8}{11}\right)
& \text{ if } g \equiv 0 \pmod{2} \\
\frac{11}{384}  (g - 1)  (g + 1)  \left(g^2 - \frac{16}{11}g - \frac{3}{11}\right)
& \text{ if } g \equiv 1 \pmod{2} 
\end{cases}.
\]

We give an example to show how to carry out the analysis in the cases where at least some of the equalities described above are satisfied.  Suppose $2b_2 = b_1$ and $a_1 = b_1+b_2$.  In this case, the inequality $2b_2 \ge g+1$ that was part of the definition of $P^*(g)$ does not necessarily hold, and instead is replaced with the new inequality $b_1+b_2 \ge g + 1$.  We have assumed that $b_1 + b_2 = a_1 > g$, so this inequality $b_1+b_2 \ge g +1$ gives no additional information.

We use Sage to compute the Hilbert series of the corresponding polyhedron as follows:\\
\texttt{\small{
PolyEx = Polyhedron(ieqs = $[(0,0,0,-1,0,1),(-1,0,0,1,-1,0),(-1,1,0,0,0,-1),\\
(-1,-1,1,0,0,0),(-1,0,-1,0,0,2),(0,0,-1,0,1,1)]$,\\
 eqns $= [(0,0,0,-1,2,0),(0,-1,0,1,1,0)]$, backend = `normaliz')\\
hilb = PolyEx.hilbert\_series($[0,0,0,0,1]$)}}.\\
This returns the Hilbert series 
 \[
\frac{t^5}{-{\left(t^{2} + t + 1\right)}^{2} {\left(t + 1\right)} {\left(t - 1\right)}^{3}},
 \]
from which we can derive a quasipolynomial expression of period $6$ for the number of integer points in this example polytope that holds for all positive integers $g$:
\[
\begin{cases}
\frac{1}{36} (g - 6)  g & \text{ if } g \equiv 0 \pmod{6} \\
\frac{1}{36}  (g - 1)^2  & \text{ if } g \equiv 1 \pmod{6} \\
\frac{1}{36} (g - 2)  (g+4) & \text{ if } g \equiv 2 \pmod{6} \\
\frac{1}{36}  (g - 3)^2  & \text{ if } g \equiv 3 \pmod{6} \\
\frac{1}{36} (g - 4)  (g+2) & \text{ if } g \equiv 4 \pmod{6} \\
\frac{1}{36}  (g + 1)^2  & \text{ if } g \equiv 5 \pmod{6} 
\end{cases}.
\]

Combining these formulas using the strategy described above completes the proof.  For readers who would like to see additional details, we have made the computations related to this part of the project available on the second author's website \footnote{\url{https://www.math.uci.edu/~nckaplan/research\_files/ordinarization}.}.
\end{proof}

\section{Ordinarization Numbers of Numerical Semigroups with Embedding Dimension $2$}\label{sec:ord_e2}

The goal of this section is to generalize the result of Bras-Amor\'os that $r\left(\langle 2,2g+1\rangle\right) = \left\lfloor\frac{g}{2}\right\rfloor$ to other families of numerical semigroups.  One natural way to generalize the family $\langle 2,2g+1\rangle$ is to consider arbitrary numerical semigroups with two minimal generators. We first recall the formula for the Frobenius number and genus of such a numerical semigroup.
\begin{proposition}\cite[Propositon 2.13]{GarciaSanchez_Rosales}\label{FG_2gen}
Let $S=\langle a,b\rangle$ be a numerical semigroup. Then,
\begin{enumerate}
\item $F(S) = ab-a-b$, and
\item $g(S) = \frac{ab-a-b+1}{2}$.
\end{enumerate}
\end{proposition}

By Proposition \ref{Prop_count_small}, we can compute $r(S)$ by counting positive elements of $S$ of size at most $g(S)$.  There is a nice geometric way to count small elements in a numerical semigroup with embedding dimension $2$.  In order to explain this approach, we introduce some additional notation related to factorizations in numerical semigroups.

\begin{defn}
Let $S$ be a numerical semigroup with minimal generating set $\{n_1,\ldots, n_e\}$.  Consider the homomorphism $\varphi\colon \Z_{\ge 0}^e \rightarrow S$ defined by 
\[
\varphi(a_1, \ldots, a_e) = a_1 n_1 +\cdots + a_e n_e.
\]  
A \emph{factorization} of $n \in S$ is an element of $\varphi^{-1}(n)$.  The \emph{set of factorizations} of $n$ is denoted by $\mathcal{Z}(n)$. The \emph{smallest Betti element} of a numerical semigroup $S$ is the smallest element of $S$ for which $|\mathcal{Z}(n)| > 1$.  
\end{defn}
\noindent For a definition of the Betti elements of a numerical semigroup and a discussion of the role that Betti elements play in the study of minimal presentations of numerical semigroups, see \cite[Chapter 7]{GarciaSanchez_Rosales} and \cite{GSOR}.

\begin{lemma}\label{lem_large_Betti}
Suppose $S$ is a numerical semigroup in which the smallest Betti element of $S$ is larger than $g(S)$.  Then $r(S)$ is equal to the total number of factorizations of positive elements $n \in S$ where $n \le g(S)$.
\end{lemma}
\begin{proof}
This follows directly from Proposition \ref{Prop_count_small} and the fact that every element of $S$ less than its smallest Betti element has a unique factorization in $S$.
\end{proof}

While it is generally not so clear how to count elements in a semigroup $S$ of size at most some bound $N$, there is a straightforward geometric approach to counting \emph{factorizations} of the elements in $S$ of size at most $N$.
\begin{lemma}\label{lem_int_points}
Let $S$ be a numerical semigroup with minimal generating set $\{n_1,\ldots, n_e\}$.  The total number of factorizations of elements $n \in S$ with $n \le N$ is equal to the number of integer points in the simplex with vertices
\[
\left(0,0,\ldots, 0\right), \left(\frac{N}{n_1},0,\ldots, 0 \right),\left(0,\frac{N}{n_2},0,\ldots, 0 \right),\ldots, \left(0,\ldots, 0,\frac{N}{n_e} \right).
\]
\end{lemma}
\begin{proof}
There is a bijection between vectors $(a_1,\ldots, a_e) \in \Z_{\ge 0}^e$ and factorizations of elements in $S$.  A factorization $(a_1,\ldots, a_e)$ corresponds to an element of size at most $N$ if and only if $\varphi(a_1,\ldots, a_e) = a_1 n_1 + \cdots + a_e n_e \le N$.  This occurs if and only if $(a_1,\ldots, a_e)$ is an integer point in the right-angled simplex defined by the inequalities $x_1,\ldots, x_e \ge 0$ and $\sum\limits_{i=1}^e x_i n_i \le N$.  This simplex has vertices as given in the statement of the lemma.
\end{proof}

In the special case where $S$ has embedding dimension $2$, this leads directly to the following characterization of $r(S)$.
\begin{cor}\label{cor:count_int_triangle}
Let $S = \langle a,b \rangle$ where $2\le a < b$ and $\gcd(a,b) = 1$.  Write $g = \frac{ab-a-b+1}{2}$.  Then $r(S)$ is equal to one less than the number of integer points in the right triangle with vertices $(0,0),\left(\frac{g}{a},0\right), \left(0,\frac{g}{b}\right)$.
\end{cor}
\begin{proof}
The smallest element of $S$ with more than one factorization is $ab = b\cdot a + 0\cdot b = 0\cdot a + a\cdot b$.  Since $g < ab$, Lemmas \ref{lem_large_Betti} and \ref{lem_int_points} imply that $r(S) = \#\{S\cap \{1,2,\ldots, g\}\}$ is equal to one less than the number of $(x,y) \in \Z_{\ge 0}^2$ satisfying $ax + by \le g$.  We subtract one from the total number of integer points in this region to account for the point $(0,0)$.  The inequalities $x,y \ge 0$ and $ax + by \le g$ define a triangle with vertices $(0,0),\left(\frac{g}{a},0\right), \left(0,\frac{g}{b}\right)$.
\end{proof}

We see that computing the ordinarization number of $\langle a,b\rangle$ is equivalent to counting integer points in a right triangle whose vertices are rational functions in $a$ and $b$.   The following result is the basis of our approach to this counting problem.
\begin{prop}\label{2genO}
Let $S = \langle a,b\rangle$ where $2\le a < b$ and $\gcd(a,b) = 1$. Write $g = \frac{ab-a-b+1}{2}$.  Then
\[
r(S) =\left\lfloor\frac{g}{b}\right\rfloor +\sum\limits_{n=0}^{\lfloor\frac{g}{b}\rfloor}\left\lfloor \frac{g-nb}{a}\right\rfloor.
\]
\end{prop}
\noindent Note that this is consistent with the result from Proposition \ref{Prop_BA_m2} which corresponds to the case where $a=2$.

\begin{proof}
We count the number of  integer points in the triangle with vertices $(0,0),\left(\frac{g}{a},0\right), \left(0,\frac{g}{b}\right)$ by counting the number on each fixed horizontal line.  On the line $y = 0$ the nonzero points $(1,0),(2,0),\ldots \left(\left\lfloor\frac{g}{a}\right\rfloor,0\right)$ correspond to the elements $a, 2a,\dots, \big\lfloor \frac{g}{a}\big\rfloor a$.  On the line $y=1$ we have the points $(0,1),(1,1),\ldots, \left(\left\lfloor\frac{g-b}{a}\right\rfloor,1\right)$ corresponding to the elements $b, a+b, 2a+b,\dots, \big\lfloor\frac{g-b}{a}\big\rfloor a+b$.  Continuing in this way, we see that for each $k \in \{1,2,\ldots, \left\lfloor\frac{g}{b}\right\rfloor\}$ there are $\left\lfloor \frac{g-kb}{a}\right\rfloor + 1$ integer points in this triangle on the line $y=k$.
\end{proof}

A careful analysis of the expression in Proposition \ref{2genO} leads to the following results.
\begin{thm}\label{thm:ord_e2}
Let $S = \langle a,b \rangle$ where $2\le a < b$ and $\gcd(a,b) = 1$.
\begin{enumerate}
\item Suppose $a$ is odd.  We have
\[
0 \le r(S) - \left( \frac{ab-a-4}{8} - \frac{b+3}{8a}\right) \le \frac{a}{4}-\frac{1}{4a}.
\]

\item Suppose $a$ is even.  We have
\[
0 \le r(S) - \left(\frac{ab-a-4}{8} \right)\le \frac{a}{4}.
\]
\end{enumerate}
\end{thm}

We also prove a result about the behavior of $r(\langle a,b\rangle)$ when $a$ is fixed and $b$ varies.
\begin{prop}\label{prop:ord_e2_quasi}
Let $a \ge 2$ be a fixed integer.  There is a linear quasipolynomial $Q_a(t)$ such that for all integers $b> a$ with $\gcd(a,b) = 1$, we have $r(\langle a,b\rangle) = Q_a(b)$.
\begin{itemize}
\item If $a$ is odd, then the period of $Q_a(t)$ divides $a$.
\item If $a$ is even, then the period of $Q_a(t)$ divides $2a$.
\end{itemize}
\end{prop}

Analyzing the expression in Proposition \ref{2genO} leads to a way to compute these quasipolynomial formulas for fixed values of $a$.  We have computed many of these formulas and list the first few of them here.
\begin{example}
In all of the formulas below, whenever we write $\langle a,b\rangle$ we require that $b>a$.
\begin{itemize}
\item We have 
\[
r(\langle 2,b\rangle) = \begin{cases}
\frac{1}{4} (b-1) & \text{ if } b \equiv 1 \pmod{4} \\
\frac{1}{4} (b-3) & \text{ if } b \equiv 3 \pmod{4}
\end{cases}.
\]

\item We have 
\[
r(\langle 3,b\rangle) = \begin{cases}
\frac{1}{3} (b-1) & \text{ if } b \equiv 1 \pmod{3} \\
\frac{1}{3} (b-2) & \text{ if } b \equiv 2 \pmod{3}
\end{cases}.
\]

\item For any odd integer $b > 4$, we have 
\[
r(\langle 4,b\rangle) = \frac{1}{2} (b-1).
\]

\item We have 
\[
r(\langle 5,b\rangle) = \begin{cases}
\frac{3}{5} (b-1) & \text{ if } b \equiv 1 \pmod{5} \\
\frac{3}{5} \left(b-\frac{1}{3}\right) & \text{ if } b \equiv 2 \pmod{5} \\
\frac{3}{5} \left(b-\frac{4}{3}\right)  & \text{ if } b \equiv 3 \pmod{5} \\
\frac{3}{5} \left(b-\frac{2}{3}\right) & \text{ if } b \equiv 4 \pmod{5} 
\end{cases}.
\]

\item We have 
\[
r(\langle 6,b\rangle) = \begin{cases}
\frac{3}{4} (b-1) & \text{ if } b \equiv 1 \pmod{12} \\
\frac{3}{4}  (b-1) & \text{ if } b \equiv 5 \pmod{12} \\
\frac{3}{4} \left(b-\frac{1}{3}\right) & \text{ if } b \equiv 7 \pmod{12} \\
\frac{3}{4} \left(b-\frac{1}{3}\right) & \text{ if } b \equiv 11 \pmod{12} 
\end{cases}.
\]

\end{itemize}
\end{example}

We now give the proof of Theorem \ref{thm:ord_e2}.  In this argument we work extensively with the \emph{fractional-part function} $\{x\} = x - \lfloor x \rfloor$.
\begin{proof}[Proof of Theorem \ref{thm:ord_e2}]
To simplify the notation, we write $g = \frac{ab-a-b+1}{2} = \frac{(a-1)(b-1)}{2}$.

We note that 
\[
\frac{g}{b} = \frac{(a-1)(b-1)}{2b} = \frac{a-1}{2}-\frac{a-1}{2b}.
\]
Since $1 \le a-1 < b$, we see that 
\[
0 < \frac{a-1}{2b} < \frac{1}{2}.
\]
This implies 
\begin{equation}\label{eqn:floor}
\left\lfloor \frac{g}{b} \right\rfloor = \begin{cases} \frac{a-3}{2} & \text{ if } a \text{ is odd} \\
\frac{a-2}{2} & \text{ if } a \text{ is even} 
\end{cases}.
\end{equation}

In order to compute $r(S)$, we need to compute the sum in Proposition \ref{2genO}.  We have
\begin{equation}\label{eqn:sums}
\sum_{n=0}^{\lfloor \frac{g}{b}\rfloor} \left\lfloor \frac{g-nb}{a} \right\rfloor = \sum_{n=0}^{\lfloor \frac{g}{b}\rfloor} \left(\frac{g-nb}{a} \right) - \sum_{n=0}^{\lfloor\frac{g}{b}\rfloor} \left\{\frac{g-nb}{a} \right\}.
\end{equation}
We consider each of these two sums separately.  Because of the difference in the expression for $\lfloor \frac{g}{b} \rfloor$ it is helpful to divide our analysis into cases based on whether $a$ is even or odd.

Suppose that $a$ is odd.  We have
\begin{align}
\sum_{n=0}^{\lfloor \frac{g}{b}\rfloor} \left(\frac{g-nb}{a} \right) & =  \sum_{n=0}^{ \frac{a-3}{2}} \left(\frac{g-nb}{a} \right) = \sum_{n=0}^{ \frac{a-3}{2}} \frac{g}{a} - \sum_{n=0}^{ \frac{a-3}{2}} \frac{nb}{a}  =  \frac{g}{a}\left( \frac{a-1}{2}\right) - \frac{b}{a} \sum_{n=0}^{ \frac{a-3}{2}} n \nonumber\\
&  =   \frac{(a-1)(b-1)}{2a} \left( \frac{a-1}{2} \right)- \frac{b}{a} \frac{\left(\frac{a-3}{2}\right) \left(\frac{a-1}{2}\right)}{2} \nonumber \\
& =  \frac{(ab - 2a + b + 2)(a - 1)}{8a}. \label{eqn:main_term_odd}
\end{align}
A similar analysis when $a$ is even shows that 
\begin{equation}\label{eqn:main_term_even}
\sum_{n=0}^{\lfloor \frac{g}{b}\rfloor} \left(\frac{g-nb}{a} \right) = \frac{ab - 2a + 2}{8}.
\end{equation}

We now turn to the next expression in the formula for $r(S)$,
\[
\sum_{n=0}^{\lfloor\frac{g}{b}\rfloor} \left\{\frac{g-nb}{a} \right\}.
\]

We first suppose that $a$ is odd. For any integer $n$, we see that $\left\{\frac{g-nb}{a}\right\}$ is a rational number equal to one of $\frac{0}{a},\frac{1}{a},\ldots, \frac{a-1}{a}$. The particular value of $\left\{\frac{g-nb}{a}\right\}$ is determined by $g-nb \pmod{a}$.

Consider the $a$ integers given by $g-nb$ where $0 \le n \le a-1$.  These values of $nb$ are distinct modulo $a$. Therefore, the integers $g-nb$ are distinct modulo $a$ as well.  We see that the $\frac{a-1}{2}$ values of $\left\{\frac{g-nb}{a}\right\}$ when $0 \le n \le \frac{a-3}{2}$ give $\frac{a-1}{2}$ distinct values in the set $\left\{\frac{0}{a},\frac{1}{a},\ldots, \frac{a-1}{a}\right\}$.  The smallest that the sum of these terms can be is 
\[
\frac{0}{a} + \frac{1}{a} + \cdots + \frac{\left(\frac{a-3}{2}\right)}{a} = \frac{(a-3)(a-1)}{8a}.
\]
The largest that the sum of these terms can be is 
\[
\frac{a-1}{a} + \frac{a-2}{a} + \cdots + \frac{a-\left(\frac{a-1}{2}\right)}{a} = 
\frac{\frac{a+1}{2}}{a}+ \cdots + \frac{a-1}{a}.
\]
This is the same as the previous sum except that $\frac{a+1}{2a} = \frac{1}{2} + \frac{1}{2a}$ has been added to each term.  Therefore,  the largest this sum can be is 
\[
\frac{(a-3)(a-1)}{8a} + \left(\frac{a+1}{2a}\right) \left( \frac{a-1}{2}\right) = \frac{(3a-1)(a-1)}{8a}.
\]
Combining these observations shows that 
\begin{equation}\label{eqn:fractional_odd}
 \frac{(a-3)(a-1)}{8a} \le \sum_{n=0}^{\lfloor\frac{g}{b}\rfloor} \left\{\frac{g-nb}{a} \right\} \le  \frac{(a-3)(a-1)}{8a} + \frac{a^2-1}{4a}.
\end{equation}
A similar analysis when $a$ is even shows that
\begin{equation}\label{eqn:fractional_even}
\frac{a-2}{8} \le \sum_{n=0}^{\lfloor\frac{g}{b}\rfloor} \left\{\frac{g-nb}{a} \right\} \le \frac{a-2}{8} + \frac{a}{4}.
\end{equation}
Combining the results of equations \eqref{eqn:floor}, \eqref{eqn:sums}, \eqref{eqn:main_term_odd}, \eqref{eqn:main_term_even}, \eqref{eqn:fractional_odd}, and \eqref{eqn:fractional_even} completes the proof.
\end{proof}

We now give the proof of Proposition \ref{prop:ord_e2_quasi}.
\begin{proof}[Proof of Proposition \ref{prop:ord_e2_quasi}]
First suppose that $a$ is odd.  In the proof of Theorem \ref{thm:ord_e2} we saw that 
\[
r(S) = \frac{a-3}{2} +  \frac{(ab - 2a + b + 2)(a - 1)}{8a} + 
\sum_{n=0}^{ \frac{a-3}{2}} \left\{\frac{g-nb}{a} \right\}.
\]
Therefore, we need only show that there is a quasipolynomial $Q_a'(t)$ of degree at most $1$ with period dividing $a$ such that for any $b>a$ with $\gcd(a,b) = 1$ we have
\[
\sum_{n=0}^{ \frac{a-3}{2}} \left\{\frac{g-nb}{a} \right\} = Q_a'(b).
\]

The sum in this expression depends on the values of $\left\{\frac{g-nb}{a} \right\}$ as $n$ ranges from $0$ to $\frac{a-3}{2}$.  This set of values depends only on the residue classes modulo $a$ represented by the integers $g-nb = \frac{a-1}{2} (b-1)- nb$.  More precisely, for each $n$ satisfying $0 \le n \le \frac{a-3}{2}$ let $r_n$ be the integer satisfying $0 \le r_n \le a-1$ and $r_n \equiv g-nb \pmod{a}$.  It is clear that 
\begin{equation}\label{eqn:sum_residues}
\sum_{n=0}^{ \frac{a-3}{2}} \left\{\frac{g-nb}{a} \right\} = \frac{r_0}{a} + \frac{r_1}{a} + \cdots + \frac{r_{\frac{a-3}{2}}}{a}.
\end{equation}
Each $r_n$ depends only on the residue class of $b$ modulo $a$, so for each residue class of $b$ modulo $a$ for which $\gcd(a,b) = 1$, the expression in \eqref{eqn:sum_residues} is fixed.  This completes the proof in this case.

Now suppose that $a$ is even.  Since we only consider integers $b$ with $b > a$ and $\gcd(a,b) = 1$, we may suppose that $b$ is odd.  We note that $g = \frac{(a-1)(b-1)}{2}$ modulo $a$ depends on the residue class of the integer $\frac{b-1}{2}$ modulo $a$, which depends on $b$ modulo $2a$.  The rest of the argument is very similar to the case where $a$ is odd.  We omit the details.
\end{proof}

We briefly address some other approaches to studying ordinarization numbers of numerical semigroups with embedding dimension $2$.
\begin{remark}
\begin{enumerate}[leftmargin=*]
\item The first potential approach also starts from Corollary \ref{cor:count_int_triangle}.  Pick's theorem gives an expression for counting integer points in lattice polygons.  The issue here is that the vertices of the right triangle in which we need to count points have \emph{rational} coordinates that are not integers.  One can apply Pick's theorem to get upper and lower bounds for the number of integer points in this triangle.  For example, it is clear that the triangle with vertices $(0,0), \left(\frac{g}{a},0\right), \left(0,\frac{g}{b}\right)$ is contained within the triangle with integer vertices $(0,0), \left(\left\lceil\frac{g}{a}\right\rceil,0\right), \left(0,\left\lceil\frac{g}{b}\right\rceil\right)$.  We pursued this approach but it did not seem to lead to estimates stronger than the ones given in Theorem \ref{thm:ord_e2}.

A more precise approach to this question comes from using the literature of integer points in rational polygons; see \cite[Section 2.7]{BeckRobins}.  Theorem 2.10 in \cite{BeckRobins} gives a formula for the number of integer points in a triangle with rational vertices in terms of certain Fourier-Dedekind sums.  It seems likely that one could use the reciprocity law of Dedekind for these sums \cite[Theorem 4]{BeckRobinsZacks} to obtain sharper estimates in many cases.  For an example of how this might work, see \cite[Section 4]{BeckRobinsZacks}.  We leave this as a direction for future exploration.

In \cite[Section 2.7]{BeckRobins} the authors explain how these formulas for counting integer points in a rational triangle lead to the result that when $P$ is a rational polygon, $i(P,n)$ is a quasipolynomial \cite[Theorem 2.11]{BeckRobins}.  This is closely related to Proposition \ref{prop:ord_e2_quasi}, but not exactly the same, since the triangles relevant to that proposition are not exactly dilates of a single fixed triangle.

\item The argument given in this section does not make use of previous results about counting factorizations of elements in $\langle a,b\rangle$.  Beck and Robins discuss the \emph{Barlow-Popoviciu formula} in \cite[Section 1.3]{BeckRobins}.  This result suggests a potential approach to computing $r(\langle a,b\rangle)$.
\begin{thm}[Barlow-Popoviciu formula]
Suppose $2\le a < b$ satisfy $\gcd(a,b) = 1$. Let $n$ be a nonnegative integer.  The number of factorizations of $n$ in $\langle a,b\rangle$ is 
\[
\frac{n}{ab} - \left\{\frac{b^{-1} n}{a} \right\} - \left\{\frac{a^{-1} n}{b} \right\} + 1,
\]
where $b^{-1} b \equiv 1 \pmod{a}$ and $a^{-1} a \equiv 1 \pmod{b}$.
\end{thm}

If $n < ab$, then it is clear that this number of factorizations is either $0$ or $1$.  Therefore, 
\[
r(\langle a,b \rangle) = \sum_{n=1}^{\frac{(a-1)(b-1)}{2}} \left(\frac{n}{ab} - \left\{\frac{b^{-1} n}{a} \right\} - \left\{\frac{a^{-1} n}{b} \right\} + 1\right).
\]
We can analyze this sum using a strategy similar to the one given in the proof of Theorem \ref{thm:ord_e2}.  This approach does not seem to lead to better results than the ones we obtained above.

Proposition \ref{2genO} divides up the integer points in the triangle described in Corollary \ref{cor:count_int_triangle} by considering points on each horizontal line.  The Barlow-Popoviciu formula gives a way to divide up the integer points in this triangle by considering points on the parallel lines $ax+by = n$ as $n$ varies.
\end{enumerate}

\end{remark}

\section{Ordinarization numbers of supersymmetric numerical semigroups}\label{sec:ord_super}

The goal of this section is to try to adapt our results on numerical semigroups with embedding dimension $2$ to a more general class of numerical semigroups.  We begin with some definitions.
\begin{defn}
A numerical semigroup $S$ is \emph{symmetric} if for any nonnegative integer $x$ we have $x \in S$ if and only if $F(S) - x \not\in S$.  
\end{defn}
\noindent It is not difficult to see that $S$ is symmetric if and only if  $g(S) = \frac{F(S)+1}{2}$ \cite[Corollary 4.5]{GarciaSanchez_Rosales}.

\begin{defn}
Let $2\le a_1 < a_2 < \ldots < a_n$ be pairwise relatively prime integers. For each $i \in \{1,2,\ldots, n\}$, define $q_i = \frac{a_1 \cdots a_n}{a_i}$.  A numerical semigroup $S$ is \emph{supersymmetric} if and only if $S = \langle q_1,\ldots, q_n\rangle$ for some choice of $a_1,\ldots, a_n$.  These semigroups are sometimes called \emph{strongly flat}, see for example \cite{StronglyFlat}.  
\end{defn}
\noindent Note that every numerical semigroup of embedding dimension $2$ is supersymmetric.

The Frobenius number of  a supersymmetric numerical semigroup is easy to describe in terms of the integers $a_1,\ldots, a_n$.  Throughout the rest of this section, when we have a collection $2\le a_1 < a_2 < \ldots < a_n$ of pairwise relatively prime integers, for each $i \in \{1,2,\ldots, n\}$ we define $q_i = \frac{a_1 \cdots a_n}{a_i}$, and we refer to $S = \langle q_1,\ldots, q_n\rangle$ as the \emph{supersymmetric numerical semigroup corresponding to $a_1,\ldots, a_n$}. 
\begin{prop}\cite{StronglyFlat}\label{Prop_genus_super}
Let $2\le a_1 < a_2 < \ldots < a_n$ be pairwise relatively prime integers and let $S$ be the supersymmetric numerical semigroup corresponding to $a_1,\ldots, a_n$.  Then
\begin{eqnarray*}
F(S) & = & (n-1)\prod\limits_{i=1}^{n}a_i-\sum\limits_{i=1}^n q_i, \\
g(S) & = & \frac{1}{2}\left(1+(n-1)\prod\limits_{i=1}^{n}a_i-\sum\limits_{i=1}^n q_i\right).
\end{eqnarray*}
\end{prop}
\noindent Note that for $n=2$ this is consistent with the formulas of Proposition \ref{FG_2gen}.  This proposition also makes it clear that supersymmetric numerical semigroups are symmetric.

Supersymmetric numerical semigroups are distinguished by the fact that they have a unique Betti element, which is $A = a_1 a_2\cdots a_n$. In particular, we have a nice formula for the smallest Betti element \cite[Example 12]{GSOR}.  This means that the only way to go between factorizations of an element in a supersymmetric semigroup is to make a sequence of trades where you replace $a_i$ copies of the generator $q_i$ with  $a_j$ copies of the generator $q_j$.  This can be made more precise with the language of minimal presentations \cite[Chapter 7]{GarciaSanchez_Rosales}.

Before discussing approximations for ordinarization numbers for supersymmetric numerical semigroups of embedding dimension larger than $2$, we briefly return to $\langle a,b \rangle$.  Consider an infinite sequence of pairs of relatively prime positive integers $(a_1,b_1), (a_2,b_2),\ldots$, where $2\le a_i < b_i$ for each $i$ and also $a_1 < a_2 < \cdots$.  The upper and lower bounds given in Theorem \ref{thm:ord_e2} imply that 
\[
\lim_{i\rightarrow \infty} \frac{r(\langle a_i, b_i \rangle)}{g(\langle a_i, b_i \rangle)} = \frac{1}{4}.
\]

We prove an analogous result for supersymmetric numerical semigroups with embedding dimension $3$.
\begin{prop}
Consider an infinite sequence of pairwise relatively prime triples of positive integers $(a_1,b_1,c_1), (a_2,b_2,c_2),\ldots$, where $2\le a_i < b_i<c_i$ for each $i$ and $a_1 < a_2 < \cdots$. Then 
\[
\lim_{i \rightarrow \infty} \frac{r(\langle a_i, b_i, c_i\rangle)}{g(\langle a_i, b_i, c_i\rangle)} = \frac{1}{6}.
\]
\end{prop}

\begin{proof}
Let $2\le a< b< c$ be pairwise relatively prime positive integers and let $S = \langle bc,ac,ab\rangle$ be the supersymmetric numerical semigroup corresponding to $a,b,c$.  Proposition \ref{Prop_genus_super} implies that 
\[
g(S) = abc - \left(\frac{ab +ac+bc-1}{2}\right).
\]
To simplify the notation we write $g = g(S)$ for the rest of the proof.

The smallest Betti element of $S$ is $abc$.  Therefore, Lemmas \ref{lem_large_Betti} and \ref{lem_int_points} imply that $r(S)$ is one less than the number of integer points in the simplex $T$ with vertices $(0,0,0), \left(\frac{g}{bc},0,0\right)$,  $\left(0,\frac{g}{ac},0\right)$, and $ \left(0,0,\frac{g}{ab}\right)$.  This is at most the number of integer points in the simplex with vertices $(0,0,0),(a,0,0),(0,b,0)$, and $(0,0,c)$ and at least one less than the number of integer points in the simplex with vertices $(0,0,0), \left(\left\lfloor\frac{g}{bc}\right\rfloor,0,0\right),  \left(0,\left\lfloor\frac{g}{ac}\right\rfloor,0\right)$, and $ \left(0,0,\left\lfloor\frac{g}{ab}\right\rfloor\right)$.  A theorem of Xu and Yau gives an exact formula for the number of integer points in a right-angled simplex in $\R^3$ with integer vertices \cite{XuYau}.  In our setting, this result implies that the number of integer points in $T$ is approximately equal to the volume of $T$, which is $\frac{1}{3!} \frac{g^3}{a^2 b^2 c^2}$.  As $a,b$, and $c$ all go to infinity, this rational function is asymptotic to $\frac{abc}{6}$.  It is now an exercise with limits to complete the proof.
\end{proof}

We would like to extend this kind of analysis to supersymmetric numerical semigroups with larger embedding dimension.  However, once the embedding dimension is at least $4$ the situation becomes more complicated.  It is no longer true that the smallest Betti element of $S$ is larger than $g(S)$.  This leads to additional difficulties, as we have to count the number of factorizations for each element of $S$ of size at most $g(S)$.

\begin{prop}\label{Prop_count_factorizations}
Let $2\le a_1 < a_2 < \ldots < a_n$ be pairwise relatively prime integers and let $S=\langle q_1,q_2,...,q_n\rangle$ be the supersymmetric numerical semigroup corresponding to $a_1,\ldots, a_n$.  Suppose $x \in S$ and let $(x_1,\ldots, x_n) \in \mathcal{Z}(x)$ be a factorization of $x$.  Then, the number of factorizations of $x$ in $S$ is
\[
|\mathcal{Z}(x)| = \binom{n+\sum\limits_{i=1}^n\big\lfloor \frac{x_i}{a_i}\big\rfloor-1}{n-1}.
\]
\end{prop}

\begin{proof}
Since $A = a_1 \cdots a_n$ is the unique Betti element of $S$, we see that 
\[
x = \left(\sum\limits_{i=1}^n\left\lfloor \frac{x_i}{a_i}\right\rfloor\right) A + \sum_{i=1}^n \left(x_i-\left\lfloor \frac{x_i}{a_i}\right\rfloor a_i \right) q_i.
\] 
It is clear that $\sum\limits_{i=1}^n \left(x_i-\left\lfloor \frac{x_i}{a_i}\right\rfloor a_i\right) q_i$ has a unique factorization in $S$.  Elements of $\mathcal{Z}(x)$ are in bijection with elements of $\mathcal{Z}\left(\left(\sum\limits_{i=1}^n\big\lfloor \frac{x_i}{a_i}\big\rfloor\right) A\right)$.  These factorizations are  in bijection with tuples $(k_1,\ldots, k_n)$ where each $k_i \ge 0$ and $k_1+\cdots + k_n = \sum\limits_{i=1}^n\big\lfloor \frac{x_i}{a_i}\big\rfloor$. 
\end{proof}
It should be possible to use this result together with the idea of counting integer points in certain right-angled simplices to approximate the ordinarization number of supersymmetric numerical semigroups with larger embedding dimension.  We leave this as a direction for future work.

We close this section by proving an analogue of Proposition \ref{2genO} for supersymmetric numerical semigroups $S = \langle q_1, \ldots, q_n\rangle$ with $n \ge 3$.  A key idea is that while it is no longer the case that each element in $S$ of size at most $g(S)$ has a unique factorization in $S$, we can still count elements in $S$ by associating each element to its factorization that uses the largest number of copies of $q_1$.
\begin{prop}\label{Theorem_ord_supersym}
Let $2\le a_1 < a_2 < \ldots < a_n$ be pairwise relatively prime integers and let $S=\langle q_1,q_2,...,q_n\rangle$ be the supersymmetric numerical semigroup corresponding to $a_1,\ldots, a_n$.  For ease of notation, we write $g = g(S)$.  The ordinarization number of $S$ is 
\begin{align*}
r(S)&=-1+\sum_{b_1=0}^{\left\lfloor\frac{g}{q_1}\right\rfloor}\sum_{b_2=0}^{\min\left\{ \left\lfloor\frac{g-b_1 q_1}{q_2}\right\rfloor,\ a_2-1\right\}}\cdots\sum_{b_n=0}^{\min\big\{ \big\lfloor\frac{g-\sum\limits_{i=1}^{n-1}b_i q_i}{q_n}\big\rfloor,\ a_n-1\big\}} 1.
\end{align*}
\end{prop}
\noindent We note that when $n=2$ this formula is equal to the formula from Proposition \ref{2genO}.

\begin{proof}
Since $A = a_1 \cdots a_n$ is the unique Betti element of $S$, each element $x\in S$ has a unique factorization that uses the largest number of the first generator $q_1$.  That is, if $x \in S$ and $(x_1,\ldots, x_n) \in \mathcal{Z}(x)$, then we identify $x$ with the factorization 
\[
\left(x_1+\sum\limits_{i=2}^n \left\lfloor\frac{x_i}{a_i}\right\rfloor a_1,x_2-\left\lfloor\frac{x_2}{a_2}\right\rfloor a_2,...,x_n-\left\lfloor\frac{x_n}{a_n}\right\rfloor a_n\right) \in \mathcal{Z}(x).
\]
Factorizations in this form are distinguished by the fact that for each $i \in \{2,\ldots, n\}$ the $i$\textsuperscript{th} entry is less than $a_i$.

We see that $(b_1, b_2,\ldots, b_n) \in \Z_{\ge 0}^n$ corresponds to a factorization of an element of $S$ of size at most $g$ if and only if $\sum\limits_{i=1}^{n} b_i q_i \le g$.  In such a factorization, we have $b_1 \le \left\lfloor \frac{g}{q_1}\right\rfloor$.  We count elements in $S$ of size at most $g$ by counting the number of choices for $(b_1, b_2,\ldots, b_n)$ satisfying $\sum\limits_{i=1}^{n} b_i q_i \le g$ and $0 \le b_i < a_i$ for each $i \in \{2,\ldots, n\}$.

We choose the entries $b_1, b_2,\ldots, b_n$ in order.  Once we have chosen $b_1$, we see that $\sum\limits_{i=1}^{n} b_i q_i \le g$ implies $b_2 \le  \left\lfloor \frac{g-b_1 q_1}{q_2}\right\rfloor$.  Continuing in this way, we see that once we have chosen $b_1,\ldots, b_j$, we must have 
\[
b_{j+1} \le \bigg\lfloor \frac{g- \sum\limits_{i=1}^{j} b_i q_j}{q_{j+1}}\bigg\rfloor.
\]
Counting the possible entries for the tuple $(b_1,\ldots, b_n)$ in this way completes the proof.
\end{proof}
One could attempt to approximate the formula in Proposition \ref{Theorem_ord_supersym} by computing the volume of the polytope in $\R^n$ defined by the inequalities $x_1, \ldots, x_n \ge 0,\ \sum_{i=1}^n q_i x_i \le g(\langle q_1,\ldots, q_n \rangle)$, and $0 \le x_i \le a_i - 1$ for each $i \in \{2,3,\ldots, n\}$.  We do not pursue this further here.

\begin{example}
Let $(a_1,a_2,a_3,a_4) = (3,5,7,11)$ and consider the corresponding supersymmetric semigroup, $S=\langle 105,165,231,385 \rangle$, where we have reordered the generators $q_1,q_2,q_3,q_4$.  Proposition \ref{Prop_genus_super} implies that 
\[
g(S) = \frac{1}{2}(1+3(1155)-105-165-231-385) = 1290.
\]
By Proposition \ref{Theorem_ord_supersym}, the ordinarization number of $S$ is 
\[
-1+\sum_{b_1=0}^{\big\lfloor\frac{1290}{105}\big\rfloor}\sum_{b_2=0}^{\min\big\{ \big\lfloor\frac{1290-105b_1}{165}\big\rfloor,6\big\}}
\sum_{b_3=0}^{\min\big\{ \big\lfloor\frac{1290-105b_1-165b_2}{231}\big\rfloor,4\big\}}
\sum_{b_4=0}^{\min\big\{ \big\lfloor\frac{1290-105b_1-165b_2-231b_3}{385}\big\rfloor,2\big\}}1 =228.
\]
\end{example}

\section{Ordinarization Numbers of Numerical Semigroups Generated by an Interval}\label{sec:interval}

The minimal generating set of every numerical semigroup with embedding dimension $2$ is given by an arithmetic progression.  The main result of this section concerns ordinarization numbers of numerical semigroups of larger embedding dimension generated by the simplest kind of arithmetic progression, an interval of positive integers.  We begin with a simple description of the elements of such a semigroup.

Let $a$ and $x$ be positive integers with $x \le a-1$ and $S = \langle a,a+1,\ldots, a+x\rangle$.  The positive elements of $S$ are exactly the union of the sets 
\[
\{a,a+1,\ldots, a+x\},\ \{2a,2a+1,\ldots, 2a+2x\},\{3a,3a+1,\ldots, 3a+3x\},\ldots.
\]
Let $n$ be the smallest positive integer such that $nx \ge a-1$.  That is, let $n = \left\lceil \frac{a-1}{x} \right\rceil$.  Since $na + nx \ge na+a-1$, we see that $S$ contains all integers greater than or equal to $na$.  The definition of $n$ implies that $(n-1)a + (n-1)x < (n-1)a+a-1$, so $F(S) = (n-1)a + (a-1)$.  Counting the number of gaps in each interval $\{ka, ka+1,\ldots, ka+a-1\}$ leads to the following formula for the genus of $S$.
\begin{lemma}\label{d1}
Let $a$ and $x$ be positive integers with $x \le a-1$.  Let $S = \langle a,a+1,\ldots, a+x\rangle$ and write $n = \left\lceil \frac{a-1}{x} \right\rceil$.   Then 
\begin{enumerate}
\item $F(S) = (n-1)a + (a-1)$, and 
\item $g(S) = na-\frac{n(n-1)}{2}x-n$.
\end{enumerate}
\end{lemma}
\noindent Keeping the notation of this lemma, we note that when $x = a-1,\ S$ is the ordinary numerical semigroup of genus $a-1$.  For the rest of this section we suppose that $x < a-1$, which means that $n \ge 2$.

Proposition \ref{Prop_count_small} implies that in order to determine $r(S)$ for a semigroup $S$ generated by an interval we need to count elements of size at most $g(S)$ in the union of intervals $\{a,a+1,\ldots, a+x\},\ \{2a,2a+1,\ldots, 2a+2x\},\{3a,3a+1,\ldots, 3a+3x\},\ldots$.  It is helpful to first determine the interval that contains $g(S)$.  This leads to cases based on the parity of $n$.
\begin{thm}\label{thm_ord_interval}
Let $a$ and $x$ be positive integers with $x < a-1$. Let $S = \langle a,a+1,\ldots, a+x\rangle$ and write $n = \left\lceil \frac{a-1}{x} \right\rceil$.  Then
\[
r(S) = 
\begin{cases}
\frac{n^2-1}{8} x + \frac{n-1}{2} & \text{ if } n \text{ is odd}\\
\frac{-n(3n-2)}{8}x+\frac{n}{2}(a-1) & \text{ if } n \text{ is even}.
\end{cases}.
\]
\end{thm}
\noindent It is a straightforward exercise to prove that when $x = 1$ this formula is consistent with the formula given in Proposition \ref{2genO}.  

\begin{proof}
To simplify the notation, throughout the proof we write $g = g(S)$.  We have $n \le \frac{a-1}{x} + \frac{x-1}{x}$, which implies $x \le \frac{a-2}{n-1}$.  Since $g= na-\frac{n(n-1)}{2}x-n$ and $\frac{a-1}{n}\leq x < \frac{a-2}{n-1}$, it follows that 
\[
na-\frac{n(a-1)}{2}-n = \frac{n(a-1)}{2}<g \leq na-\frac{(n-1)(a-1)}{2}-n= \frac{(a-1)(n+1)}{2}.
\]

We first consider the case where $n$ is odd.  Note that $\frac{na-n}{2}>\frac{na-a}{2}$.  We see that 
\[
\left(\frac{n-1}{2} \right) a < g < \left(\frac{n+1}{2}\right) a.
\]
We first count the number of elements of $S$ in each interval $\{ka,ka+1,\ldots, ka+(a-1)\}$ where $0 \le k < \frac{n-1}{2}$.  The total number of these elements is $\sum\limits_{\ell=1}^{\frac{n-3}{2}}(\ell x+1)$.  

We will prove that $g > \frac{n-1}{2} a + \frac{n-1}{2} x$ and therefore $g$ is a gap of $S$.  This will imply that the number of elements of $S$ in the interval 
\[
\left\{\frac{n-1}{2} a, \frac{n-1}{2} a +1,\ldots, \frac{n-1}{2} a + \frac{n-1}{2} x\right\}
\]
that are of size at most $g$ is $\frac{n-1}{2} x +1$.  Simplifying 
\[
\sum\limits_{\ell=1}^{\frac{n-1}{2}}(\ell x+1) = \frac{n^2-1}{8} x + \frac{n-1}{2} 
\]
completes the proof in this case.

We see that 
\[
g = na - \frac{n(n-1)}{2} x - n > \frac{n-1}{2} a + \frac{n-1}{2} x 
\]
if and only if 
\[
\frac{n+1}{2} a - n > \frac{(n+1)(n-1)}{2} x.
\]
The right-hand side is clearly a nondecreasing function of $x$.  Since $x \le \frac{a-2}{n-1}$, we see that 
\[
\frac{(n+1)(n-1)}{2} x \le \frac{(n+1)(a-2)}{2} < \frac{n+1}{2} a - n,
\]
completing the proof that $g > \frac{n-1}{2} a + \frac{n-1}{2} x$.

Now suppose that $n$ is even.   Since $x \le \frac{a-2}{n-1}$, we have
\[
g = na - \frac{n(n-1)}{2} x - n \ge na - \frac{n(a-2)}{2} - n = \frac{n}{2} a.
\]
Since $x \ge \frac{a-1}{n}$, we have
\[
g \le \frac{(a-1)(n+1)}{2} < \left(\frac{n}{2}+1\right) a. 
\]
We follow the same strategy as in the case where $n$ is odd.  We first count elements of $S$ in each interval $\{ka,ka+1,\ldots, ka+(a-1)\}$ where $0 \le k \le \frac{n}{2}-1$.  The total number of these elements is $\sum\limits_{\ell=1}^{\frac{n}{2}-1}(\ell x+1)$.  

We will prove that $g \le \frac{n}{2} a + \frac{n}{2} x$ and therefore $g$ is not a gap of $S$. This will imply that the number of elements of $S$ in the interval 
\[
\left\{\frac{n}{2} a, \frac{n}{2} a +1,\ldots, \frac{n}{2} a + \frac{n}{2} x\right\}
\]
that are of size at most $g$ is $g - \frac{n}{2} a + 1$.  Simplifying 
\[
\left(\sum\limits_{\ell=1}^{\frac{n}{2}-1}(\ell x+1) \right)+ g - \frac{n}{2} a + 1
=
\frac{(n-2)n}{8} x + \left(\frac{n}{2}-1\right) + \left(na-\frac{n(n-1)}{2} x - n\right) - \frac{n}{2} a + 1
\]
completes the proof in this case.

We see that 
\[
g = na - \frac{n(n-1)}{2} x - n \le 
 \frac{n}{2} a + \frac{n}{2} x
\]
if and only if 
\[
\frac{n}{2} a - n \le \frac{n^2}{2} x.
\]
Dividing by $n$, this holds if and only if 
\[
\frac{a}{2} - 1 \le \frac{n}{2} x.
\]
The right-hand side is clearly a nondecreasing function of $x$.  Since $x \ge \frac{a-1}{n}$, we have 
\[
\frac{n}{2} x \ge \frac{a-1}{2}  \ge \frac{a}{2}  - 1,
\]
completing the proof that $g \le \frac{n}{2} a + \frac{n}{2} x$.
\end{proof}

We highlight some simple cases of this result.
\begin{example}
Suppose that $S = \langle a,a+1,\ldots, a+x\rangle$ where $\frac{a-1}{2}\le x < a-1$.  This is the case where $n = 2$ in the statement of Theorem \ref{thm_ord_interval}. We have 
\[
g(S) = a-1 + (a-1-x)-1 = 2a-x-2 < 2a.
\]  
The number of elements in the interval $\{a,a+1,\ldots, a+x\}$ that are at most $g(S)$ is $a-x-1$.  Therefore, $r(S) = a-x-1$, which is consistent with the formula given in Theorem \ref{thm_ord_interval} for $n=2$.  
\end{example}

\begin{example}
Suppose that $S=\langle a,a+1,...,a+x\rangle$ where $\frac{a-1}{3}\leq x< \frac{a-1}{2}$. This is the case where $n=3$ in the statement of Theorem \ref{thm_ord_interval}.  We can count elements in $\{a,a+1,\ldots, a+x\}$ and $\{2a,2a+1,\ldots, 2a+2x\}$ of size at most $g(S)$ directly, or just apply the formula in Theorem \ref{thm_ord_interval} to see that $r(S) = x+1$.
\end{example}

\section*{Acknowledgments}
The authors thank Christopher O'Neill for helpful discussions related to the material in Section \ref{sec:ng2}.  They thank Maria Bras-Amor\'os, Matthias Beck, and Christopher O'Neill for helpful comments on an earlier draft of this paper.  The second author was supported by NSF grant DMS 2154223.

\end{document}